%% file: main.tex
\pdfoutput=1
\documentclass[10pt,a4paper]{amsart}

\calclayout

\usepackage[utf8]{inputenc}
    \usepackage[T1]{fontenc}
    \usepackage[sc]{mathpazo}
    \usepackage{stackengine,amsmath,amsthm,amssymb,mathtools,soul,wrapfig}
    \usepackage{mathdots}
    \usepackage{graphicx}
    \usepackage{tikz,tikz-cd}
            \tikzset{
                subseteq/.style={
                draw=none,
                edge node={node [sloped, allow upside down, auto=false]{$\subseteq$}}},
                Subseteq/.style={
                draw=none,
                every to/.append style={
                edge node={node [sloped, allow upside down, auto=false]{$\subseteq$}}}
                }
            }
            \tikzset{
                equal/.style={
                draw=none,
                edge node={node [sloped, allow upside down, auto=false]{$=$}}},
                Equal/.style={
                draw=none,
                every to/.append style={
                edge node={node [sloped, allow upside down, auto=false]{$=$}}}
                }
            }
            \tikzset{
            rotated_label/.style={anchor=south, rotate=90, inner sep=.5mm}
            }
    \usetikzlibrary{decorations}
    \usepackage{verbatim}
    \usepackage{booktabs}
    \usepackage{array}
    \usepackage{tabularx}
    \usepackage{booktabs}
    \usepackage{siunitx}
    \usepackage{rotating}
    \usepackage{adjustbox}
    \usepackage{longtable}
    \usepackage{caption} 
    \usepackage{stmaryrd}
    \usepackage{nicefrac}
    \usepackage{enumitem}
    \usepackage{xfrac}
    \usepackage[textwidth=0.8in,textsize=tiny]{todonotes}
\usepackage[backref=true,giveninits=true,doi=false,url=false,isbn=false,backend=biber]{biblatex}
    \newbibmacro{string+doiurlisbn}[1]{%
      \iffieldundef{doi}{%
        \iffieldundef{url}{%
          \iffieldundef{isbn}{%
            \iffieldundef{issn}{%
              #1%
            }{%
              \href{http://books.google.com/books?vid=ISSN\thefield{issn}}{#1}%
            }%
          }{%
            \href{http://books.google.com/books?vid=ISBN\thefield{isbn}}{#1}%
          }%
        }{%
          \href{\thefield{url}}{#1}%
        }%
      }{%
        \href{http://dx.doi.org/\thefield{doi}}{#1}%
      }%
    }
    \DeclareFieldFormat{title}{\usebibmacro{string+doiurlisbn}{\mkbibemph{#1}}}
    
    \def\CalQ{{\mathcal{Q}}}
    \def\CalH{{\mathcal{H}}}
    
    \DeclareMathOperator*{\Res}{Res}
    \DeclareFieldFormat[article,incollection,thesis,misc,inproceedings,online,inbook]{title}{\usebibmacro{string+doiurlisbn}{\mkbibquote{#1}}}
    \DeclareFieldFormat{year}{\usebibmacro{year-parenthesis}{\mkbibemph{#1}}}
\renewbibmacro*{date}{%
\iffieldundef{year}
{}
{\ifentrytype{misc}{\printtext[parens]{\printdate}}{\printdate}}}%
\usepackage{hyperref}
    \hypersetup{colorlinks=true,allcolors=blue}
    \usepackage[capitalize,nameinlink,noabbrev]{cleveref}

\usepackage{amsthm}
    \theoremstyle{plain}
        \newtheorem{theorem}{Theorem}[section]
        \newtheorem{lemma}[theorem]{Lemma}
        \newtheorem{proposition}[theorem]{Proposition}

    \theoremstyle{definition}

    \theoremstyle{remark}

        \newtheorem{remark}[theorem]{Remark}

\input{shortcuts.tex}

\makeatletter
\@namedef{subjclassname@2020}{\textup{2020} Mathematics Subject Classification}
\makeatother

\title{On diophantine properties for values of Dedekind zeta functions}
    \author{Jerson Caro}
    \address{Jerson Caro - Department of Mathematics \& Statistics, Boston University, 665 Commonwealth Avenue, Boston, MA 02215, USA}
    \email{\href{mailto:jlcaro@bu.edu}{jlcaro@bu.edu}}
    \author{Fabien Pazuki}
    \address{Fabien Pazuki - Department of Mathematical Sciences, University of Copenhagen, Universitetsparken 5,
    2100 Copenhagen, Denmark}
    \email{\href{mailto:fpazuki@math.ku.dk}{fpazuki@math.ku.dk}}
    
    \author{Riccardo Pengo}
    \address{Riccardo Pengo - Dipartimento di Scienze Matematiche e Informatiche, Scienze Fisiche e Scienze della Terra, Università degli Studi di Messina, Viale Ferdinando Stagno d'Alcontres, 31, 98166 Messina, Italy}
    \email{\href{mailto:riccardo.pengo@unime.it}{riccardo.pengo@unime.it}}

\allowdisplaybreaks
\begin{document}
    \begin{abstract}
        We study the Northcott and Bogomolov property for special values of Dedekind $\zeta$-functions at real values $\sigma \in \mathbb{R}$. 
        We prove, in particular, that the Bogomolov property is not satisfied when $\sigma \geq \frac{1}{2}$. If $\sigma > 1$, we produce certain families of number fields having arbitrarily large degrees, whose Dedekind $\zeta$-functions $\zeta_K(s)$ attain arbitrarily small values at $s = \sigma$. 
        On the other hand, if $\frac{1}{2} \leq \sigma \leq 1$, we construct suitable families of quadratic number fields, employing either Soundararajan's resonance method, which works when $\frac{1}{2} \leq \sigma < 1$, or results on random Euler products by Granville and Soundararajan, and by Lamzouri, which work when $\frac{1}{2} < \sigma \leq 1$. 
        We complete the study by proving that the Dedekind $\zeta$ function together with the degree satisfies the Northcott property for every complex $s\in{\mathbb{C}}$ such that $\mathrm{Re}(s) <0$, generalizing previous work of Généreux and Lalín.
    \end{abstract}

\maketitle

    \section{Introduction}
    
    Let $\mathcal{N}$ denote the set of isomorphism classes of number fields. For every $s \in \mathbb{C}$, we define the \textit{Dedekind height}:
    \begin{equation} \label{eq:Dedekind_height}
        \begin{aligned}
            h_s \colon \mathcal{N} &\to \mathbb{R}_{\geq 0} \\
            [K] &\mapsto \lvert \zeta_K^\ast(s) \rvert,
        \end{aligned}
    \end{equation}
    where $\zeta^\ast_K(s)$ denotes the \textit{special value} of the Dedekind zeta function $\zeta_K(z)$ at $z = s$, which is classically defined as the first non-vanishing coefficient appearing in the Laurent series expansion of the meromorphic function $\zeta_K(z)$ around the point $z = s$, as we recall in \cref{sec:notation_dedekind}.
    We are particularly interested in studying the topology of the set $h_s(\mathcal{N}) \subseteq \mathbb{R}_{\geq 0}$, and in particular to know if $h_s$ satisfies the \textit{Northcott property} (\textit{i.e.} if for each $B \geq 0$ the set $\{ [K] \in \mathcal{N} \colon h_s([K]) \leq B \}$ is finite) or the \textit{Bogomolov property} (\textit{i.e.} if the infimum of $h_s(\mathcal{N})$ is isolated, and whether this infimum is a minimum). 
    \subsection{Historical remarks}
    When $s = n$ is an integer, the function $h_n$ was considered by the second and third-named authors of the present paper \cite[Theorem~1.2]{Pazuki_Pengo_2024}, who proved that $h_n$ has the Northcott property if $n \leq 0$, and does not have this property when $n \geq 1$.
    In particular, when $n = 0$ this result is an immediate consequence of the celebrated Brauer-Siegel theorem \cite[Theorem~2]{Brauer_1947}, which compares the size of $\zeta^\ast_K(0)$ to the size of the discriminant of $K$.

    Généreux and Lalín \cite{GeLa24} also considered the function $\mathcal{N} \to \mathbb{R}_{\geq 0}$ given by $[K] \mapsto \lvert \zeta_K(s) \rvert$, which coincides with $h_s$ for every $s \in \mathbb{C}$ such that either $\mathrm{Re}(s) > 1$ or $\mathrm{Re}(s) < 0$ and $s \not\in \mathbb{Z}_{< 0}$.
    They have proven that this function does not have the Northcott property when $\mathrm{Re}(s) > \frac{1}{2}$, but has this property for $s$ lying in certain regions of the half-plane $\mathrm{Re}(s) < 0$.

More generally, the last two authors of the present paper developed a research framework around the Northcott and Bogomolov properties of special values of $L$-functions \cite{Pazuki_Pengo_2024}. From their first results, it seems that, for some natural families of $L$-functions, the Northcott property holds at integers on the left of the critical strip, and does not hold at integers on the right of the critical strip. What happens when $s \in \mathbb{C}$ is not an integer, and in particular when $s$ lies inside the critical strip, remains intriguing. 

In the special case of the Dedekind zeta functions associated with function fields in positive characteristic, this analysis was extended to special values taken at any complex point $s \in \mathbb{C}$ by Généreux, Lalín, and Li \cite{GeLaLi22}. Moreover, G\'en\'ereux and Lalín \cite{GeLa24} considered the Northcott property for the special values at any complex number $s \in \mathbb{C}$ of the Dedekind zeta functions associated to number fields, and proved in particular in \cite[Theorem~1.6]{GeLa24} that the Northcott property for the function $h_s$ defined in \eqref{eq:Dedekind_height} is not satisfied in the region $\frac{1}{2}<\mathrm{Re}(s)<1$. Furthermore, they showed that the validity of the Generalized Riemann Hypothesis for the Dedekind zeta functions $\zeta_{\mathbb{Q}(\sqrt{p})}(s)$ associated to real quadratic fields of prime discriminant implies that for every $B > 0$ and every $s \in \mathbb{C}$ such that $\frac{1}{2} < \mathrm{Re}(s) < 1$ the set $h_s(\mathcal{N}) \cap [0,B]$ is infinite, and in particular that $h_s$ does not satisfy the Bogomolov property in the region $\frac{1}{2} < \mathrm{Re}(s) < 1$. Let us note once again that these results do not deal with the heights $h_s$ associated with complex numbers $s \in \mathbb{C}$ lying on the critical line $\mathrm{Re}(s)=\frac{1}{2}$, which however plays a special role, as the center of symmetry of all the functional equations satisfied by the Dedekind zeta functions $\zeta_K(s)$, which we recall in \eqref{functional equation}. 

The aforementioned results, and in particular the interest in studying the difference between the values and the \textit{special} values of Dedekind zeta functions, date back a long time. Indeed, in a short paper published in the first volume of Acta Arithmetica in 1935, Chowla \cite{Cho35} indicates that one can expect a non-vanishing result for certain values of the $L$-functions associated with non-trivial real characters $\chi$. More precisely, one can expect that $L(s,\chi)>0$ for any $s \in \mathbb{R}_{> 0}$, and Chowla proves a criterion to obtain this result. Chowla returns to this question in page xv and page 82 of \cite{Cho65}, pointing out that, thirty years after his original work \cite{Cho35}, this question was still open. We refer the interested reader to \cite{Lou23} for a more recent study. These results of Chowla help compare special values $\zeta_K^\ast(s)$ with the actual values $\zeta_K(s)$, and therefore can be viewed as precursors of the present study, in a way.

\subsection{Main results}
At present, we focus on the Northcott and Bogomolov properties \textit{on the real line}. Concerning the Bogomolov property, we obtain the following theorem.

\begin{theorem}[Main Theorem]\label{Main theorem}
Let $\sigma \in [\frac{1}{2},+\infty)$.
    The function $h_\sigma \colon \mathcal{N} \to \mathbb{R}_{\geq 0}$ defined by $h_\sigma([K]) := \lvert \zeta^\ast_K(\sigma) \rvert$ does not have the Bogomolov property. More precisely, for every $\sigma \in [\frac{1}{2},+\infty)$ and every $B \in (\inf(h_\sigma(\mathcal{N})),+\infty)$ the intersection $h_\sigma(\mathcal{N}) \cap (0,B)$ is infinite.
\end{theorem}

The proof is divided into several tasks corresponding to the four intervals $\{\frac{1}{2}\}$ (the center of the critical strip), $(\frac{1}{2},1)$ (the interior of the critical strip), $\{1\}$ (the boundary of the critical strip), and $(1,+\infty)$ (the region of absolute convergence), each requiring different methods of study. More precisely:
\begin{itemize}
\item on the critical line, \textit{i.e.} when $\sigma=\frac{1}{2}$, we prove \cref{Main theorem} by studying the second moment of the resonator coefficients constructed by Soundararajan \cite{Soundararajan_2008} for $L$-functions attached to quadratic characters;
\item inside the critical strip, but not on the critical line, \textit{i.e.} when $\sigma \in (\frac{1}{2},1)$, we prove \cref{Main theorem} either by adapting Soundararajan's resonator method or by employing previous work of Lamzouri \cite{Lamzouri_2011}, concerning the similarities between the distribution of values of quadratic Dirichlet $L$-functions and random Euler products;
\item on the boundary of the critical strip, \textit{i.e.} when $\sigma = 1$, we use previous work of Granville and Soundararajan \cite{Granville_Soundararajan_2003}, which proves once again that special values of quadratic $L$-functions at $\sigma = 1$ are distributed like random Euler products;
\item in the region of absolute convergence, \textit{i.e.} when $\sigma \in (1,+\infty)$, we prove \cref{Main theorem} by constructing explicit families of number fields $\{ K_d \}_{d \in \mathbb{N}}$ such that $[K_d \colon \mathbb{Q}] \to +\infty$ and $\zeta_{K_d}(\sigma) \to 1$ as $d \to +\infty$. We do so by imposing that many rational primes stay inert in each field $K_d$.
\end{itemize}

Let us be more precise. To prove \cref{Main theorem} in $\sigma=\frac{1}{2}$, we focus on quadratic fields $\mathbb{Q}(\sqrt{d})$, where $d$ is a fundamental discriminant. By Artin's formalism, for every $s \in \C$ and every fundamental discriminant $d$ we have that
\begin{equation} \label{Artin formalism}
    \zeta_{\mathbb{Q}(\sqrt{d})}(s) = \zeta(s) \cdot L(s,\chi_d),    
\end{equation}
where $\zeta(s) = \zeta_\mathbb{Q}(s)$ denotes Riemann's zeta function, while $L(s,\chi_d)$ denotes the $L$-function associated to the quadratic Dirichlet character $\chi_d$, whose definition is recalled in \cref{notation}. This decomposition shows that, in order to prove that $h_{\frac{1}{2}}$ does not have the Bogomolov property, it suffices to show that the image of the function $d \mapsto \left\lvert L\left( \frac{1}{2},\chi_d\right) \right\rvert$, defined over the set of fundamental discriminants, accumulates towards zero. 
This is precisely what we do in \cref{sec:cricial_center}, where we prove the following theorem:
\begin{theorem}[Main Theorem for $\sigma=\frac{1}{2}$]\label{Bogomolov}
There exists an absolute constant $c_0$ such that for every $X \geq c_0$ there exists a fundamental discriminant $d \in \mathbb{N}$ such that $X \leq d \leq 2X$ and 
\begin{equation} \label{eq:Bogomolov_inequalities}
    0<\left\lvert L\left(\frac{1}{2},\chi_d\right)\right\rvert\leq \exp\left(-
\left(
\frac{1}{\sqrt{5}}+o(1)\right)\sqrt{\frac{\log X}{\log\log X}}\right),
\end{equation}
where $\chi_d$ denotes the real primitive character associated with the fundamental discriminant $d$. In particular $\inf(h_{\frac{1}{2}}(\mathcal{N})) = 0$, and the function $h_s$ defined in \eqref{eq:Dedekind_height} does not have the Bogomolov property when $s = \frac{1}{2}$.
\end{theorem}

This theorem implies the existence of a family of fundamental discriminants whose associated Dirichlet $L$-functions attain strictly positive values at $\sigma = \frac{1}{2}$ that tend to $0$. This implies that $0$ is the infimum and an accumulation point of the image of $h_{\frac{1}{2}}$. Therefore, this function does not have the Bogomolov property.

The strategy to prove \cref{Bogomolov} is based on the resonators method used by Soundararajan \cite{Soundararajan_2008}. Our main contribution is provided by \cref{SecondMoment}, where we compute the second moment associated with a resonator function.

\begin{remark}
Soundararajan already studied the function $d \mapsto L\left( \frac{1}{2}, \chi_d \right)$ in his work \cite{Soundararajan_2008}. However,
\cite[Theorem 2]{Soundararajan_2008} only implies that for $X$ sufficiently large, there exists a fundamental discriminant $d \in \mathbb{Z}$ such that $X\leq |d|\leq 2X$ and
\begin{equation} \label{eq:sound_claim}
    L\left(\frac{1}{2},\chi_d\right)\leq \exp\left(-
\left(
\frac{1}{\sqrt{5}}+o(1)\right)\sqrt{\frac{\log X}{\log\log X}}\right).
\end{equation}
While Soundararajan originally expressed this bound with the absolute value $\lvert L\left(\frac{1}{2},\chi_d\right) \rvert$  on the left-hand side of \eqref{eq:sound_claim}, the authors of the present paper noticed that the formulation without the absolute value is the correct implication of \cite[Theorem 2]{Soundararajan_2008}. Furthermore, Soundararajan suggested in private communication that the authors of the present paper explore the second moment associated with the resonator function to strengthen the analysis. Soundararajan also noted that assuming Chowla's conjecture would lead to \cref{Bogomolov}. Building on these insights, we unconditionally prove \cref{Bogomolov}, as explained in \cref{sec:cricial_center}.
\end{remark}

In \cref{sec:cricialright}, we extend the resonator method for every $\sigma$ in the interval $\left(\frac{1}{2},1\right)$. This allows us to prove an analogous result to \cref{Bogomolov}, which, by \eqref{Artin formalism}, implies \cref{Main theorem} in this interval, that is, that $h_\sigma$ does not satisfy the Bogomolov property for $\sigma \in \left(\frac{1}{2},1\right)$.

\begin{theorem}[Main Theorem for $\frac{1}{2} < \sigma < 1$]\label{Bogomolov2}
For every $\sigma \in \left(\frac{1}{2},1\right)$ there exists a constant $c_1(\sigma)$ such that for every $X \geq c_1(\sigma)$ there exists a fundamental discriminant $d \in \mathbb{N}$ such that $X \leq d \leq \frac{5}{2} X$ and 
\[
0<|L(\sigma,\chi_d)|\leq \exp\left(-
\left(
4\sigma+o(1)\right)\frac{\log^{\frac{1-\sigma}{2\sigma}} X}{\log\log X}\right),
\]
where $\chi_d$ denotes the real primitive character associated with the fundamental discriminant $d$. 
In particular, for every real number $\sigma \in \left( \frac{1}{2}, 1 \right)$, we have that $\inf(h_{\sigma}(\mathcal{N})) = 0$, and the function $h_s$ defined in \eqref{eq:Dedekind_height} does not have the Bogomolov property when $s = \sigma$.
\end{theorem}

\begin{remark}
The natural questions concerning the effectivity of the bounds provided by \cref{Bogomolov} and \cref{Bogomolov2} still have to be explored.
\end{remark}

When $\sigma = 1$, we employ a result of Granville and Soundararajan \cite{Granville_Soundararajan_2003}, recalled in \cref{prop:Granville_Soundararajan}, to show that $h_\sigma = h_1$ does not have the Bogomolov property.
Finally, if $\sigma \in (1,\infty)$, we have that $\zeta_K^\ast(\sigma) = \zeta_K(\sigma) \geq 1$ for every number field $K$. In \cref{bigger 1}, we construct a family of number fields depending on $\sigma$, such that the values of the Dedekind zeta function at $\sigma$ tend to $1$.

\begin{theorem}[Main Theorem for $\sigma > 1$]\label{s>1}
        For every $\sigma \in (1,+\infty)$ there exist two families of number fields $\{ K_{d,\sigma}, K_{d,\sigma}' \}$ such that $\zeta_{K_{d,\sigma}}(\sigma) \to 1$ and $\zeta_{K_{d,\sigma}'}(\sigma) \to +\infty$ as $d \to +\infty$. In particular, the function $h_s$ defined in \eqref{eq:Dedekind_height} does not have the Bogomolov property when $s = \sigma$ is a real number such that $\sigma > 1$.
    \end{theorem}

Let us now turn to the Northcott property. G\'en\'ereux, Lalín, and Li \cite{GeLaLi22} considered the analogous problem for isomorphism classes of function fields $K$ with constant base field $\F_q$. In particular, they proved that $h_s$ has the Northcott property when $\mathrm{Re}(s)<0$. In contrast, G\'en\'ereux and Lalín \cite{GeLa24} demonstrated that the Dedekind zeta function (not its special value, simply its value at the considered point, but note that the two coincide for $\mathrm{Re}(s) < 0$ and $s \not\in \mathbb{Z}$) exhibits different behavior for $\mathrm{Re}(s)<0$. Specifically, they showed that close to the negative integers, $h_s$ does not have the Northcott property, whereas far from these points, it does. In this context, we observe that $h_s$ has the Northcott property if one bounds furthermore the degrees of the number fields involved.

\begin{theorem}\label{Re negative}
Let $s$ be a complex number with $\mathrm{Re}(s)< 0$. Then, the function
\[
[K]\mapsto \left(\lvert\zeta_K^\ast(s) \rvert, [K \colon \mathbb{Q}] \right)
\]
has the Northcott property.
\end{theorem}

In \cref{negative s}, we prove \cref{Re negative} by directly analyzing the functional equation satisfied by this function.

\begin{remark}
    In fact, the proof of \cref{Re negative} shows that if $\lvert \zeta_K^\ast(s) \rvert$ is bounded then the \textit{root discriminant} $\delta_K := \lvert \Delta_K \rvert^{\frac{1}{[K \colon \mathbb{Q}]}}$ is also bounded. This entails that the only families of number fields violating the Northcott property of the Dedekind height $h_s$ when $\mathrm{Re}(s) < 0$ are built from towers of number fields having bounded root discriminant, as those considered by Généreux and Lalín in \cite{GeLa24}.
\end{remark}

\section{Notation}\label{notation}

The aim of the present section is to gather several pieces of notation that will be used in the main body of the present paper.

\subsection{Numbers} We let $\mathbb{N} = \{0,1,2,\dots\}$ denote the set of natural numbers, $\mathbb{Z}$ denote the ring of rational integers, $\mathbb{Q}$ denote the field of rational numbers, $\mathbb{R}$ denote the field of real numbers and $\mathbb{C}$ denote the field of complex numbers.
Given a complex number $s \in \mathbb{C}$ we let $\sigma = \mathrm{Re}(s)$ be its real part, and we let $t = \mathrm{Im}(s)$ be its imaginary part, so that $s = \sigma + i t$, where $i$ denotes the imaginary unit.

\subsection{Asymptotics}
Let $S$ be a topological space and $f,g \colon S \to \mathbb{C}$ be two functions. Then:
\begin{itemize}
    \item given $s_0 \in S$, we write $f(s) = O(g(s))$ for $s \to s_0$, or equivalently $f(s) \ll g(s)$ for $s \to s_0$, if and only if there exists an open subset $U \subseteq S$ such that $s_0 \in U$, and a real number $C>0$ such that $|f(s)| \leq C |g(s)|$ for every $s \in U$;
    \item given $s_0 \in S$, we write $f(s) = o(g(s))$ for $s \to s_0$ if there exists an open subset $U \subseteq S$ such that $s_0 \in U$, and a function $\varepsilon \colon U \to \mathbb{C}$, such that $f(s) = \varepsilon(s) g(s)$ for every $s \in U$, and $\displaystyle{\lim_{s \to s_0} \varepsilon(s) = 0}$;
    \item given $s_0 \in S$, we write $f(s) \sim g(s)$ for $s \to s_0$ if there exists an open subset $U \subseteq S$ such that $s_0 \in U$, and a function $h \colon U \to \mathbb{C}$, such that $f(s) = g(s) (1 + h(s))$ for $s \in U$, and $h(s) = o(1)$ for $s \to s_0$.
\end{itemize}

When $S = \mathbb{R}$ or $S = \mathbb{C}$, these definitions readily generalize to include $s_0 = \infty$. Moreover, if $f,g \colon \mathbb{N} \to \mathbb{C}$ then $f(n) \ll g(n)$ as $n \to +\infty$ if and only if there exists $C > 0$ such that $\lvert f(n) \rvert \leq C \lvert g(n) \rvert$ for every $n \in \mathbb{N}$. In this case, we will avoid writing $n \to +\infty$ and simply write $f(n) \ll g(n)$.

\subsection{Constants}

We say that a real number $c \in \mathbb{R}$ appearing in some expression that bounds a function $f$ is an \textit{absolute constant} if it does not depend on any of the variables or parameters featured in the definition of $f$. In the same context, we will say that a polynomial $P \in \mathbb{R}[x]$ appearing in some expression bounding $f$ is an \textit{absolute polynomial} if none of its coefficients depends on any of the variables or parameters featured in the definition of $f$.

\subsection{Special functions}
We denote Euler's gamma function by $\Gamma(s)$. This can be defined by the infinite product 
\[
\Gamma(s) := \frac{1}{s} \prod_{n=1}^{+\infty} \left[ \frac{n}{s+n} \left( 1 + \frac{1}{n} \right)^s \right], \] and can be used to define the functions
\[
\begin{aligned}
\Gamma_{\mathbb{R}}(s) &:= \pi^{-\frac{s}{2}} \Gamma\left(\frac{s}{2}\right) \\
\Gamma_{\mathbb{C}}(s) &:= (2\pi)^{-s} \Gamma(s),
\end{aligned}
\]
where $\pi := \left(\Gamma\left(\frac{1}{2}\right)\right)^2$. These functions can, in turn, be used to define the function
\begin{equation} \label{eq:Gamma_m}
    \Gamma_{m}(s) :=
    \min\left\{
    \left|\frac{\Gamma_\R(1 - s)}{\Gamma_\R(s)} \right|,\left|\frac{\Gamma_\C(1 - s)^
{\frac{1}{2}}}{\Gamma_\C(s)^{\frac{1}{2}}}\right|\right\}
\end{equation}
featured in \cite{GeLa24}.
Moreover, we let $e^s = \exp(s) := \sum_{n=1}^{+\infty} \frac{s^n}{n!}$ denotes the complex exponential, where $n! := \Gamma(n+1)$ denotes the factorial of a natural number $n \in \mathbb{N}$. We also let $\log s$ denote the principal determination of the complex logarithm, whose base is Napier's number $e := \exp(1)$, and we let $\cosh(s) := \frac{e^s+e^{-s}}{2}$ and $\tanh(s) := \frac{e^s-e^{-s}}{e^s+e^{-s}}$ denote the hyperbolic cosine and tangent.

We will also need certain non-standard special functions, such as the function $W_2$ defined by the following integral:
\begin{equation} \label{eq:W2}
    W_2(\xi) := \frac{1}{2\pi i}\int_{(c)} \left(\frac{\Gamma(\frac{w}{2}+\frac{1}{4})}{\Gamma(\frac{1}{4})}\right)^2 \xi^{-w} \frac{dw}{w},
\end{equation}
for every $\xi > 0$, where $(c)$ denotes the vertical line going from $c - i \infty$ to $c + i \infty$, and $c$ is any positive real number.
By \cite[Lemma 2.1]{Soundararajan_2000} we know that $W_2$ is smooth and satisfies $W_2(\xi) = 1 + O(\xi^{\frac{1}{2}-\varepsilon})$ for small $\xi$, and $W_2(\xi) \ll \exp(-\xi)$ for large $\xi$.
Finally, $\Phi$ denotes the following function
\begin{equation}
    \label{eq:Phi}
    \Phi(x) =\begin{cases}
\exp \left(\frac{1}{(2x-1)(x-3)}
\right)
&\text{if }\frac{1}{2}
< x < 3,\\
0 &\text{otherwise,}
\end{cases}
\end{equation}
while $A$ denotes the function
\begin{equation}
    \label{eq:A}
    A_{\alpha_1,\alpha_2}(l) = \sum_{m\text{ odd}} \frac{1}{m} \left( \sum_{
n_1 n_2=l_1 m^2}
\frac{1}{n_{1}^{\alpha_{1}}n_{2}^{\alpha_{2}}} \right)
\prod_{\substack{p \ \text{prime} \\ p \mid lm}}\left(\frac{p}{p+1}\right)
\end{equation}
which depends on two variables $\alpha_1$ and $\alpha_2$ and on a natural number $l = l_1 l_2^2$, where $l_1$ is squarefree and coprime to $l_2$.

We will also need a special function $\eta$ defined by an Euler product. More precisely, for $\alpha\in \C$ with $\mathrm{Re}(\alpha)>\frac{1}{2}$ and for an integer $l=l_1l_2^2$, with $l_1$ squarefree and coprime to $l_2$, we let 
\begin{equation}
    \label{eq:eta}
    \eta(\alpha;l):= \prod_{p \ \text{prime}}\eta_{p}(\alpha;l)
\end{equation}
be the function introduced in \cite[Lemma~5.1]{Soundararajan_2000}, whose Euler factors $\eta_p(\alpha;l)$ are given by
\begin{equation} \label{eq:eta_p}
    \eta_{p}(\alpha;l) := \begin{cases}
(1-2^{-\alpha})^3, & \text{if} \ p = 2 \\ 
1-\frac{3}{p^{\alpha}(p+1)}-\frac{p-3}{p^{2\alpha}(p+1)}-\frac{1}{p^{3\alpha}(p+1)} & \text{if} \ p \geq 3 \ \text{and} \ p\nmid l\\
\left(\frac{p}{p+1}\right)\left(1-\frac{1}{p^\alpha}\right) & \text{if} \ p \geq 3 \ \text{and} \  p\mid l_1\\
\left(\frac{p}{p+1}\right)\left(1-\frac{1}{p^{2\alpha}}\right) & \text{if} \ p \geq 3 \ \text{and} \ p\mid l \text{ but }p\nmid l_1.
\end{cases}
\end{equation}
We will also need the quotients
\begin{equation} \label{eq:Gp}
    G_p(\alpha;l) := \frac{\eta_p(\alpha;l)}{\eta_p(\alpha;1)},
\end{equation}
and their derivatives with respect to $\alpha$, denoted $G_p^{(i)}(\alpha;l)$. We will prove in \eqref{eq:Gp_asymptotics_main} that
\begin{equation} \label{eq:Gp_asymptotics_notation}
    \frac{G_p^{(i)}(1;l)}{G_p(1;l)} = \begin{cases}
        c_2(i) \frac{\log^i(p)}{p} \mathcal{H}_{1,i}(p), \ &\text{if} \ p \mid l_1, \\
        c_3(i) \frac{\log^i(p)}{p^2} \mathcal{H}_{2,i}(p), \ &\text{if} \ p \mid l_2,
    \end{cases}
\end{equation}
for every $i \in \mathbb{Z}_{\geq 1}$, where $c_2(i)$ and $c_3(i)$ are constants depending only on $i$, and $\mathcal{H}_{1,i}(p), \mathcal{H}_{2,i}(p)$ are functions such that 
\[
\mathcal{H}_{1,i}(p),\mathcal{H}_{2,i}(p) = 1 + O\left( \frac{1}{p} \right),
\]
as $p \to +\infty$.

\subsection{Arithmetic functions}
\label{sec:multiplicative_functions}

Recall that a function $f \colon \mathbb{N} \to \mathbb{C}$ is \textit{multiplicative} if $f(ab) = f(a) f(b)$ for every $a,b \in \mathbb{N}$ such that $\mathrm{gcd}(a,b) = 1$, where $\mathrm{gcd}$ denotes the greatest common divisor. Furthermore, we say that $f$ is \textit{completely multiplicative} if $f(ab) = f(a) f(b)$ for every $a,b \in \mathbb{N}$. Analogously, a function $f \colon \mathbb{N} \to \mathbb{C}$ is \textit{additive} if $f(ab) = f(a) + f(b)$ for every $a,b \in \mathbb{N}$ such that $\gcd(a,b) = 1$, and $f$ is \textit{completely additive} if $f(ab) = f(a) + f(b)$ for every $a,b \in \mathbb{N}$.

In the present paper, we let $\mu$ denote the \textit{Möbius function}, the unique multiplicative function defined on prime powers as
\begin{equation} \label{eq:mobius}
    \mu(p^k) := \begin{cases}
        1, \ &\text{if} \ k = 0 \\
        -1, \ &\text{if} \ k = 1 \\
        0, \ &\text{if} \ k \geq 2
    \end{cases}
\end{equation}
for every rational prime $p$ and every integer $k \geq 0$. Moreover, we let $\tau$ denote the number of divisors function, which is the multiplicative function defined as
\begin{equation} \label{eq:tau}
    \tau(n) := \# \{ m \in \mathbb{N} \colon m \mid n \} = \sum_{m \mid n} 1
\end{equation}
for every natural number $n \in \mathbb{N}$. Furthermore, we let $\sigma$ denote the \textit{sum of divisors} function, which is the multiplicative function defined as
\begin{equation} \label{eq:sigma}
    \sigma(n) := \sum_{m \mid n} m
\end{equation}
for every natural number $n \in \mathbb{N}$. Finally, we will let $\omega$ be the additive function that counts the number of distinct prime factors of an integer, defined as
\[
    \omega(n) := \# \{p \ \text{prime} \colon p \mid n \} = \sum_{\substack{p \ \text{prime} \\ p \mid n}} 1 
\]
for every $n \in \mathbb{N}$.

We will also need a number of non-standard arithmetic functions. First of all, we denote by $h$ the unique multiplicative function defined as
\begin{equation} \label{eq:h}
h(p^k)=1+\frac{1}{p}+\frac{1}{p^2}-\frac{4}{p(p+1)},
\end{equation}
for every rational prime $p$ and every $k \in \mathbb{N}$. Moreover, given $D > 0$ we denote by $\epsilon_D$ the multiplicative function defined as 
\begin{equation}
    \label{eq:epsilon}
    \epsilon_D(n) := D^{\omega(n)}
\end{equation}
for every integer $n \in \mathbb{N}$.

We need as well a completely multiplicative function $f$ that depends on two fixed real numbers $L > 0$ and $\sigma \in \left[\frac{1}{2},1 \right)$. When $\sigma = \frac{1}{2}$, we define
\begin{equation} \label{eq:f_center}
    f(p) := \begin{cases}
        \frac{L}{\sqrt{p} \log p}, \ &\text{if} \ L^2 \leq p \leq \exp(\log^2X), \\
        0, \ &\text{otherwise}
    \end{cases}
\end{equation}
for every rational prime $p$. On the other hand, when $\frac{1}{2} < \sigma < 1$, we define
\begin{equation} \label{eq:f_right}
    f(p) := \begin{cases}
        L p^{-\sigma}, \ &\text{if} \ p \geq L^{1/\sigma} \\
        0, \ &\text{otherwise}
    \end{cases}
\end{equation}
for every rational prime $p$. Since $f$ is completely multiplicative, these definitions determine it completely.

We need three more multiplicative functions $F, G$ and $H$, defined as
\begin{align}
    F(p^k) &:= 1+\frac{f(p)^2}{h(p)}-4\frac{f(p) \sqrt{p}}{h(p) (p+1)} \label{eq:F} \\
    G(p^k) &:= \frac{\log p}{p^{2}}\frac{f(p)^2}{h(p)F(p)} \label{eq:G} \\
    H(p^k) &:= -4 \log^k p\frac{f(p)\sqrt{p}}{h(p) (p+1) F(p)} \label{eq:H},
\end{align}
for every rational prime $p$ and every $k \in \mathbb{N}$. 
These functions are defined in analogy with \cite[page 481]{Soundararajan_2008}. Moreover, the functions $G$ and $H$ are supported on natural numbers whose prime factors lie in the interval $[L^2,\exp(\log^2L)]$, where $L$ is the same parameter used in the definition of $f$. Finally, we will need the arithmetic function $M$ defined as
\begin{equation} \label{eq:M}
    M(n)=\frac{1}{\omega(n)}\left(\sum_{\substack{p \ \text{prime} \\ p \mid n}}\frac{\CalH_{1,1}(p)}{p}\right) H(n),
\end{equation}
for every integer $n \in \mathbb{N}$, where $\mathcal{H}_{1,1}$ is the function appearing in \eqref{eq:Gp_asymptotics_notation}.

\subsection{Dedekind zeta functions}
\label{sec:notation_dedekind}
Let $K$ be a number field. The \textit{Dedekind zeta function} $\zeta_K(s)$ is defined for complex numbers $s = \sigma + i t$ with $\sigma > 1$ by the series:
\begin{equation}\label{dfn Dedekind}
\zeta_K(s) = \sum_{\mathfrak{a}} \frac{1}{\mathrm{N}(\mathfrak{a})^s},
\end{equation}
where the sum is over all nonzero ideals $\mathfrak{a}$ of the ring of integers $\mathcal{O}_K$ in $K$, and $\mathrm{N}(\mathfrak{a})$ denotes the index $[\mathcal{O}_K\colon\mathfrak{a}]$. 
Let $\Delta_K$ denote the absolute discriminant of $K$, and let $r_1$ and $r_2$ denote the number of real and complex places of $K$, respectively. 
The Dedekind zeta function $\zeta_K(s)$ extends to a meromorphic function on the complex plane with a simple pole at $s = 1$, and satisfies the functional equation
\begin{equation}\label{functional equation}
\zeta_K(s) = \zeta_K(1 - s)
\left( \frac{\Gamma_\R(1 - s)}{\Gamma_\R(s)} \right)^{r_1} \left( \frac{\Gamma_\C(1 - s)}{\Gamma_\C(s)} \right)^{r_2} |\Delta_K|^{\frac{1}{2}-s}.
\end{equation}
The following function
\[
\zeta_K^{*}(s) := \lim_{z \to s} \frac{\zeta_K(z)}{(z-s)^{\mathrm{ord}_{z=s}(\zeta_K(z))}},
\]
which assigns to every complex number $s \in \mathbb{C}$ the \textit{special value} of the Dedekind zeta function $\zeta_K$ at $s$, plays the main role in the present paper.

Finally, given a quadratic fundamental discriminant $d \in \mathbb{Z}$, we let $\chi_d \colon \mathbb{Z} \to \{\pm 1\}$ be the quadratic Dirichlet character defined as 
\[
    \chi_d(n) := \left( \frac{d}{n} \right)
\]
for every $n \in \mathbb{Z}$, where $\left( \frac{d}{\cdot} \right)$ denotes the Kronecker symbol. We can associate to $\chi_d$ the \textit{Dirichlet $L$-function} $L(s,\chi_d)$, defined for complex numbers $s$ with $\mathrm{Re}(s) > 1$ by the following series:
\begin{equation}\label{dfn L Dirichlet}
L(s,\chi_d) = \sum_{n=1}^\infty \frac{\chi_d(n)}{n^s}.
\end{equation}
This Dirichlet series extends to an entire function $L(s,\chi_d)$, defined for every $s \in \mathbb{C}$. By employing the Artin formalism, we obtain the following relation for a quadratic field $K$ with discriminant $\Delta_K = d$:
\[
\zeta_K(s) = \zeta(s) \cdot L(s,\chi_d),
\]
where $\zeta(s) := \zeta_\mathbb{Q}(s)$ denotes the Riemann zeta function.

We will, in particular, be interested in several moments associated with these $L$-functions associated with Dirichlet characters. For instance, we will study the asymptotic behavior of the function
\begin{equation}
\label{eq:M_Sono}
M_{\alpha_1,\alpha_2}(l;X) := \sum_{d=1}^{+\infty} \mu(2d)^2 L\left( \frac{1}{2} + \alpha_1,\chi_{8d} \right) L\left( \frac{1}{2} + \alpha_2,\chi_{8d} \right) \chi_{8d}(l) \Phi\left( \frac{d}{X} \right)
\end{equation}
when $X \to +\infty$. This function $M_{\alpha_1,\alpha_2}(l;X)$ was introduced by Sono in \cite{sono2020second}, and depends on two parameters $\alpha_1$ and $\alpha_2$. Moreover, the notation $\Phi$ featured in \eqref{eq:M_Sono} refers to the function $\Phi$ defined in \eqref{eq:Phi}. In particular, the sum appearing in \eqref{eq:M_Sono} is always finite.

\subsection{More on constants}

In this paper, we will express several results in terms of some constants, denoted by $c_j$ for some natural number $j \in \mathbb{N}$, which reflects the order with which these constants appear in the main text of the present paper. Some of these constants will be absolute, while some others will depend on certain parameters. Moreover, some of these constants will remain unspecified, while some others will be explicitly determined. More precisely, the notation $c_0,c_1(\sigma),c_2(i),c_3(i),c_4,c_7,c_8(\sigma,\delta),c_9(\delta),c_{11},c_{12},c_{13}(\sigma),c_{14}(\sigma),c_{15}(\sigma),c_{16}(\sigma),c_{17}(\sigma),c_{18}(\sigma)$ and $c_{19}(\sigma) = c_{18}(\sigma) \zeta(2 \sigma)^3$ will be used for some unspecified constants, depending on the parameters in brackets. On the other hand, we will also need a couple of explicit constants, given by the following formulas
\begin{align}
    \label{eq:c6}
    c_6 &:= \frac{1}{8}\prod_{\substack{p \ \text{prime}\\p\geq 3}}\left(1-\frac{1}{p}\right)h(p) = \frac{1}{8}\prod_{p\geq 3}\left(1-\frac{1}{p}\right) \left( 1+\frac{1}{p}+\frac{1}{p^2}-\frac{4}{p(p+1)} \right) = 0.068586928786\dots \\
    \label{eq:c5}
    c_5 &:= \frac{c_6}{576 \zeta(2)} = \frac{c_6}{96 \pi^2} = 0.000072388633\dots \\
    \label{eq:c10} c_{10}(\delta) &:= \frac{e^{\frac{1}{2}}}{4 \pi^2} \int_{-\infty}^{+\infty} \lvert \Gamma(it-\delta) \rvert dt \\
\label{eq:c20}
c_{20}(\sigma) &:= \sigma^{\frac{2 \sigma}{1-\sigma}} \cdot (1-\sigma)^{\frac{2\sigma-1}{\sigma-1}} \cdot \left(\int_0^{+\infty} \log(\cosh(x)) x^{-\frac{1}{\sigma}} \frac{dx}{x} \right)^{\frac{\sigma}{\sigma-1}} \\
\label{eq:21}
c_{21} &:= \exp\left(\int_1^{+\infty} (1-\tanh(x)) \frac{dx}{x} - \int_0^1 \tanh(x) \frac{dx}{x}\right) = 0.440969247215\dots,
\end{align}
where the product in \eqref{eq:c6} runs over all the odd primes.
We will also need some classical mathematical constants, such as $\pi$ and $e$, introduced above, as well as the Stieltjes constants $\gamma_j$, which can be defined as 
\begin{equation} \label{eq:Stieltjes}
    \gamma_j := -\frac{1}{j+1} \sum_{n=0}^{+\infty} \frac{n!}{n+1} \sum_{k=0}^n \frac{(-1)^k \log^{j+1}(k+1)}{k! (n-k)!}
\end{equation}
for every $j \in \mathbb{N}$. Note in particular that $\gamma_0$ is the Euler-Mascheroni constant.


    \section{On the critical line}
    \label{sec:cricial_center}
    The aim of this section is to employ the resonance method of Soundararajan \cite{Soundararajan_2008} to show that the function 
    \[
        \begin{aligned}
            h_{\frac{1}{2}}\colon \mathcal{N} &\to \mathbb{R}_{\geq 0} \\
            [K] &\mapsto \left\lvert \zeta_K^\ast\left(\frac{1}{2}\right) \right\rvert
        \end{aligned}
    \]
    does not have the Bogomolov (or Northcott) property, thereby proving \cref{Bogomolov}.
    To do so, let us first define 
    \begin{equation}\label{resonator}
     R(8d) := \sum_{l \leq N} r(l) \chi_{8d}(l),   
    \end{equation}
     where $N > 0$ is a parameter to be specified later. This function $R$ is called a ``resonator'', of the kind used in \cite{Soundararajan_2008}. 
The coefficients $r(l)$ are to be chosen in such a way that $R(8d)$ ``resonates'' with the large values of $L\left(\frac{1}{2},\chi_{8d}\right)$. 
More precisely, we will take 
\begin{equation} \label{eq:resonator_coefficients}
    r(l) := \mu(l) f(l),    
\end{equation}
where $\mu$ is the M\"obius function, introduced in \eqref{eq:mobius}, and $f$ is the multiplicative function introduced in \eqref{eq:f_center}, where the parameter $L$ will be specified later. This choice of $r$ comes from \cite[Proposition 3.3]{Soundararajan_2008}, inspired by \cite[Theorem 2.1]{Soundararajan_2008}. 

The resonator $R$ can be used to define the \textit{resonated moments} of the sequence of $L$-values $L\left(\frac{1}{2},\chi_{8d}\right)$. The first and second of these moments are defined as
\begin{align}
    M_1(R,X) &= \sum_{\frac{X}{16}\leq d\leq \frac{X}{8}}\mu(2d)^2R(8d)^2L\left(\frac{1}{2},\chi_{8d}\right), \label{MRX} \\
    M_2(R,X) &= \sum_{\frac{X}{16}\leq d\leq \frac{X}{8}} \mu(2d)^2 R(8d)^2 L\left(\frac{1}{2},\chi_{8d}\right)^2,\label{M2RX}
\end{align}
in the same spirit as \cite[Section 3]{Soundararajan_2008}. 
The asymptotic behavior of these two moments when $X \to +\infty$ can be determined explicitly if one chooses the parameters $L$ and $N$ correctly. For the first moment $M_1(R,X)$ this has been done by Soundararajan, who obtained the following result.  

\begin{theorem}[Soundararajan]
\label{FirstMoment} 
Let $X > 0$ be large, $N\leq X^{\frac{1}{20}-\varepsilon}$ be large, and set $L=\sqrt{\log N\log\log N}$. Let $f$ be defined as in \eqref{eq:f_center}, $r$ be defined as in \eqref{eq:resonator_coefficients}, and $R$ be defined as in \eqref{resonator}. Then, as $X \to +\infty$ we have that
\begin{equation}
M_1(R,X)\sim c_4 X(\log X)\prod_{p \ \text{prime}}\left(1+f(p)^2-2\frac{f(p)}{\sqrt{p}}\right),
\end{equation}
for some positive absolute constant $c_4$, where $M_1(R,X)$ is defined as in \eqref{MRX}.
\end{theorem}

On the other hand, the following asymptotic for the second resonated moment $M_2(R,X)$ is the main technical result of the present section.

\begin{theorem}\label{SecondMoment}
Let $X > 0$ be large, $N\leq X^{\frac{1}{20}-\varepsilon}$ be large, and set $L=\sqrt{\log N\log\log N}$. Let $f$ be defined as in \eqref{eq:f_center}, $r$ be defined as in \eqref{eq:resonator_coefficients}, and $R$ be defined as in \eqref{resonator}. Then, as $X \to +\infty$ we have that
\begin{equation}
M_2(R,X)\sim c_5 X(\log X)^3\prod_{p \ \text{prime}}\left(1+f(p)^2-4\frac{f(p)}{\sqrt{p}}\right),
\end{equation}
where $c_5$ is the absolute constant defined in \eqref{eq:c5} and $M_2(R,X)$ is defined as in \eqref{M2RX}.
\end{theorem}

Before we proceed with the proof of \cref{SecondMoment}, let us see how this result can be combined with \cref{FirstMoment} in order to obtain a proof of \cref{Bogomolov}.

\begin{proof}[Proof of \cref{Bogomolov}]
Suppose by contradiction that there exists some $X > 0$ large such that for no fundamental discriminant $d$ with $X \leq d \leq 2 X$, the two inequalities portrayed in \eqref{eq:Bogomolov_inequalities} hold true.
This is obviously equivalent to assuming that there exists some large $X > 0$ such that for every fundamental discriminant $d$ with $\frac{X}{16} \leq d \leq \frac{X}{8}$ we have that either \[
\left\lvert L\left( \frac{1}{2}, \chi_{8 d} \right) \right\rvert > \exp\left( -\left( \frac{1}{\sqrt{5}} + o(1) \right) \sqrt{\frac{\log X}{\log\log X}} \right) \]
or $L\left( \frac{1}{2}, \chi_{8 d} \right) = 0$. Under this assumption, we would have that
\begin{equation} \label{eq:absurd_bound}
    \displaystyle{\sum_{\frac{X}{16}\leq d\leq \frac{X}{8}}\mu(2d)^2R(8d)^2L\left(\frac{1}{2},\chi_{8d}\right)^2} > \exp\left(- \left(\frac{2}{\sqrt{5}}+o(1)\right)\sqrt{\frac{\log X}{\log\log X}}\right) 
\displaystyle{\sum_{\substack{\frac{X}{16}\leq d\leq \frac{X}{8}\\ L\left(\frac{1}{2},\chi_{8d}\right)\neq 0 }}\mu(2d)^2R(8d)^2},
\end{equation}
which would negate the asymptotics provided by \cref{FirstMoment,SecondMoment}. To see this, we will use the elementary inequality
\begin{equation} \label{eq:Titu}
    \frac{a_1^2}{b_1} + \dots + \frac{a_n^2}{b_n} \geq \frac{(a_1+\dots+a_n)^2}{b_1+\dots+b_n},
\end{equation}
which is valid for any real numbers $a_1,\dots,a_n$ and any positive real numbers $b_1,\dots,b_n$. Note that \eqref{eq:Titu} follows easily from the Cauchy-Schwarz inequality and is known under the names of Sedrakyan's inequality, Engel's form, Bergström's inequality, or Titu's lemma.
Applying \eqref{eq:Titu}, we see that:
\[
\displaystyle{\sum_{\substack{\frac{X}{16}\leq d\leq \frac{X}{8}\\ L\left(\frac{1}{2},\chi_d\right)\neq 0 }}\mu(2d)^2R(8d)^2}\geq \frac{\displaystyle{\left(\sum_{\frac{X}{16}\leq d\leq \frac{X}{8}}\mu(2d)^2R(8d)^2L\left(\frac{1}{2},\chi_{8d}\right)\right)^2}}{\displaystyle{\sum_{\frac{X}{16}\leq d\leq \frac{X}{8}}\mu(2d)^2R(8d)^2\Big|L\left(\frac{1}{2},\chi_{8d}\right)\Big|^2}}=\frac{M_1(R,X)^2}{M_2(R,X)},
\]
which implies that
\begin{equation}\label{principal2}
\frac{\displaystyle{\sum_{\frac{X}{16}\leq d\leq \frac{X}{8}}\mu(2d)^2R(8d)^2 L\left(\frac{1}{2},\chi_d\right) ^2}}{\displaystyle{\sum_{\substack{\frac{X}{16}\leq d\leq \frac{X}{8}\\ L\left(\frac{1}{2},\chi_d\right)\neq 0 }}\mu(2d)^2R(8d)^2}}\leq \left(\frac{M_2(R,X)}{M_1(R,X)}\right)^2. 
\end{equation}
Combining \eqref{principal2} with \eqref{eq:absurd_bound} we see that if the conclusion of \cref{Bogomolov} were false we would have that
\begin{equation}
    \label{eq:absurd_bound2}
    \left(\frac{M_2(R,X)}{M_1(R,X)}\right)^2 > \exp\left(- \left(\frac{2}{\sqrt{5}}+o(1)\right)\sqrt{\frac{\log X}{\log\log X}}\right)
\end{equation}
for some $X > 0$ large.
This, however, contradicts the conclusions of \cref{FirstMoment,SecondMoment}, which imply that
\[
    \left( \frac{M_2(R,X)}{M_1(R,X)} \right)^2 \sim \frac{c_5^2}{c_4}(\log^4 X)\prod_{p \ \text{prime}}\left(\frac{1+f(p)^2-4\frac{f(p)}{\sqrt{p}}}{1+f(p)^2-2\frac{f(p)}{\sqrt{p}}}\right)^2
\]
for every $X > 0$ large, and for every $N \leq X^{\frac{1}{20} - \varepsilon}$ which is also large.
Indeed, notice that
\begin{equation} \label{eq:main_theorem_from_moments}
\begin{split}\raisetag{-35ex}
(\log^4 X)\prod_{p \ \text{prime}}\left(\frac{1+f(p)^2-4\frac{f(p)}{\sqrt{p}}}{1+f(p)^2-2\frac{f(p)}{\sqrt{p}}}\right)^2
&=(\log^4 X)\prod_{p \ \text{prime}} \left(1-\frac{2\frac{f(p)}{\sqrt{p}}}{1+f(p)^2-2\frac{f(p)}{\sqrt{p}}}\right)^2\\
&=\exp\left( 4 \log\log X + 2 \sum_{p \ \text{prime}} \log\left( 1 - \frac{2 \frac{f(p)}{\sqrt{p}}}{1+f(p)^2-2\frac{f(p)}{\sqrt{p}}} \right) \right) \\
&= \exp\left( 4 \log\log X - (4+o(1)) \sum_{p \ \text{prime}} \frac{\frac{f(p)}{\sqrt{p}}}{1+f(p)^2-2\frac{f(p)}{\sqrt{p}}} \right) \\
&=\exp\left(4 \log\log X -(4+o(1))\sum_{p \text{ prime}}\frac{f(p)}{\sqrt{p}}\right)\\
&=
\exp \left(4 \log\log X -(4+o(1))\sum_{\substack{p \ \text{prime} \\ L^2\leq p\leq \exp(\log^2 L)}}\frac{L}{p\log p}\right)\\
&=
\exp \left(4 \log\log X -(4+o(1))\sqrt{\log N\log\log N}\sum_{\substack{p \ \text{prime} \\ L^2\leq p\leq \exp(\log^2 L)}}\frac{1}{p\log p}\right)\\
&=\exp\left(4 \log\log X -(4+o(1))\sqrt{\frac{\log N}{\log\log N}}\right), 
\end{split}
\end{equation}
where the last equality follows from the prime number theorem.
Given that this computation will recur in several forthcoming results, we take this first occurrence as an opportunity to present it in full detail. More precisely, let $\pi(t) := \#\{ p \ \text{prime} \colon p \leq t \}$ be the prime counting function, let $\mathrm{Li}(t) := \int_2^{t} \frac{du}{\log u}$ be the offset logarithmic integral function, and let $R(t) := \pi(t) - \mathrm{Li}(t)$. Then, we have that
\begin{equation} \label{eq:RS_integrals}
    \begin{aligned}
    \sum_{\substack{p \ \text{prime} \\ L^2 \leq p \leq \exp(\log^2L)}} \frac{1}{p \log p} &= \int_{L^2}^{\exp(\log^2L)} \frac{d\pi(t)}{t \log t} = \int_{L^2}^{\exp(\log^2L)} \frac{d\mathrm{Li}(t)}{t \log t} + \int_{L^2}^{\exp(\log^2 L )} \frac{dR(t)}{t \log t} \\
    &= \int_{L^2}^{\exp(\log^2 L)} \frac{dt}{t \log^2 t} + \int_{L^2}^{\exp(\log^2L)} \frac{dR(t)}{t \log t}
    \end{aligned}
\end{equation}
where the integrals appearing in these formulas are taken in the sense of Riemann-Stieltjes. Now, an elementary calculation shows that
\begin{equation} \label{eq:PNT1_main}
    \int_{L^2}^{\exp(\log^2L)} \frac{dt}{t \log^2t} = \frac{1}{2 \log L} - \frac{1}{\log^2L} \ll \frac{1}{\log\log N},
\end{equation}
because $L = \sqrt{\log N \log\log N}$.
Moreover, if we apply integration by parts to the last integral appearing in \eqref{eq:RS_integrals} we see that
\begin{equation} \label{eq:RS_integrals_parts}
\int_{L^2}^{\exp(\log^2L)} \frac{dR(t)}{t \log t} = \frac{R(\exp(\log^2L))}{\exp(\log^2L) \log^2L} - \frac{R(L^2)}{2 L^2 \log L} + \int_{L^2}^{\exp(\log^2L)} R(t) d\left( -\frac{1}{t \log t} \right).
\end{equation}
Now, the prime number theorem asserts that $R(t) \ll_A \frac{t}{\log^A t}$ for every $A > 0$, as explained for example in \cite[Section~2.4]{Iwaniec_Kowalski_2004}. Using this upper bound, we see easily that
\begin{equation} \label{eq:PNT1_remainder1}
    \frac{R(\exp(\log^2L))}{\exp(\log^2L) \log^2L} - \frac{R(L^2)}{2 L^2 \log L} \ll \frac{1}{\log^2L} \ll \frac{1}{\log\log^2N} \ll \frac{1}{\log\log N}
\end{equation}
Finally, we can apply the prime number theorem again to see that
\begin{equation} \label{eq:PNT1_remainder2}
    \begin{aligned}
        \int_{L^2}^{\exp(\log^2L)} R(t) d\left( -\frac{1}{t \log t} \right) &= \int_{L^2}^{\exp(\log^2L)} \frac{R(t) (\log t+1)}{t^2 \log^2t} dt \\ &\ll \int_{L^2}^{\exp(\log^2L)} \frac{\log t+1}{t^2 \log^3t} dt \\
        &= \frac{1}{4 \log^2 L} + \frac{1}{2 \log L} - \frac{1}{2 \log^4L} - \frac{1}{\log^2L} \\
        &\ll \frac{1}{\log\log N}.
    \end{aligned}
\end{equation}
Plugging \eqref{eq:PNT1_remainder1} and \eqref{eq:PNT1_remainder2} into \eqref{eq:RS_integrals_parts}, we now see that
\begin{equation} \label{eq:PNT1_remainder}
    \int_{L^2}^{\exp(\log^2(L))} \frac{dR(t)}{t \log t} \ll \frac{1}{\log\log N}.        
\end{equation}
Moreover, plugging \eqref{eq:PNT1_remainder} and \eqref{eq:PNT1_main} into \eqref{eq:RS_integrals} we finally see that
\[
    \sum_{\substack{p \ \text{prime} \\ L^2 \leq p \leq \exp(\log^2(L))}} \frac{1}{p \log p} \ll \frac{1}{\log\log N},
\]
which implies the last equality portrayed in \eqref{eq:main_theorem_from_moments}.
Taking $N = X^{\frac{1}{20}-\varepsilon}$ we have
\begin{equation*}
\begin{aligned}
    \left(\frac{M_2(R,X)}{M_1(R,X)}\right)^2 &= \exp\left(4 \log\log X -(4+o(1))\sqrt{\frac{\log N}{\log\log N}}\right) \\ &= \exp\left(-(4+o(1))\sqrt{\frac{\log N}{\log\log N}}\right) \\ &=
\exp\left(-\left(
\frac{2}{\sqrt{5}}+o(1)\right)\sqrt{\frac{\log X}{\log\log X}}\right),
\end{aligned}
\end{equation*}
for every $X > 0$ large. This shows that \eqref{eq:absurd_bound2} cannot be true, and therefore concludes the proof of \cref{Bogomolov}.
\end{proof}

We use the rest of this section to prove \cref{SecondMoment}, following Soundararajan's methods from \cite{Soundararajan_2000} and \cite{Soundararajan_2008}.

\begin{proof}[Proof of \cref{SecondMoment}]
By \cite[Lemma 2.2]{Soundararajan_2000}, we have the following expression
\begin{equation}\label{eq1}
L\left(\frac{1}{2},\chi_{8d}\right)^2=2\sum_{n=1}^{\infty}\frac{\tau(n)}{\sqrt{n}} \cdot \chi_{8d}(n) \cdot W_2\left(\frac{n \pi}{8 d}\right),
\end{equation}
where $\tau$ denotes the number of divisors function, defined in \eqref{eq:tau}, while $W_2$ denotes the function introduced in \eqref{eq:W2}.
Substituting the explicit expression for $L\left(\frac{1}{2},\chi_{8d}\right)^2$ given by \eqref{eq1} and the explicit definition of $R(8d)$ given in \eqref{resonator} inside the definition of $M_2(R,X)$ given in \eqref{MRX} we see that
\[
M_2(R,X)=2\sum_{n_1,n_2\leq N}r(n_1)r(n_2)\sum_{n=1}^{\infty}\frac{\tau(n)}{\sqrt{n}}\sum_{\frac{X}{16}\leq d\leq \frac{X}{8}}\mu(2d)^2\chi_{8d}(nn_1n_2)W_2\left(\frac{n\pi}{8d}\right).
\]
Recall that $W_2$ is smooth and satisfies $W_2(\xi) = 1 + O(\xi^{\frac{1}{2}-\varepsilon})$ for small $\xi$, and $W_2(\xi) \ll \exp(-\xi)$ for large $\xi$, as proven in \cite[Lemma 2.1]{Soundararajan_2000}.
This implies that
\begin{equation}\label{equ2}
\begin{aligned}
M_2(R,X)&=\frac{X}{8\zeta(2)}\sum_{n_1,n_2\leq N}r(n_1)r(n_2)\sum_{\substack{n\\nn_1n_2=\text{odd square}}}\frac{\tau(n)}{\sqrt{n}}\prod_{\substack{p \ \text{prime} \\ p\mid 2nn_1n_2}}\left(\frac{p}{p+1}\right)\int_1^2 W_2\left(\frac{2n\pi}{Xt}\right)dt \\ &+ O\left(X^{\frac{7}{8}+\varepsilon}N^{\frac{1}{2}} \left(\sum_{n\leq N}|r(n)|\right)^2\right),
\end{aligned}
\end{equation}
as one can see by an application of the Pólya-Vinogradov inequality \cite[Theorem~12.5]{Iwaniec_Kowalski_2004}, analogous to the proof of \cite[Lemma 3.2]{Soundararajan_2008}.

For $N\leq X^{\frac{1}{20}-\varepsilon}$ we clearly have that $O\left(X^{\frac{7}{8}+\varepsilon}N^{\frac{1}{2}} \left(\sum_{n\leq N}|r(n)|\right)^2\right) = O(X^{1-\varepsilon})$ and therefore \eqref{equ2} implies that the main term of $M_2(R,X)$ is
\begin{equation}\label{equ3}
\frac{X}{12\zeta(2)}\sum_{\substack{a,r,s\\ar,as\leq N\\(a,r)=(a,s)=(r,s)=1}}
\mu(a)^2f(a)^2\frac{\mu(r)f(r)\mu(s)f(s)}{\sqrt{rs}} \sum_{m \text{ odd}}\frac{\tau(rsm^2)}{m}\prod_{\substack{p \ \text{prime} \\ p\mid arsm}} \left(\frac{p}{p+1}\right)\int_1^2 W_2\left(\frac{2rsm^2\pi}{Xt}\right)dt,
\end{equation}
which is obtained from \eqref{equ2} by replacing the three variables $(n_1,n_2,n)$ with the variables $(a,r,s,m)$, where we set $a = \mathrm{gcd}(n_1,n_2)$ and $r = n_1/a$, while $s = n_2/a$ and $m$ is the unique odd natural number such that $n_1 n_2 n = a^2 r^2 s^2 m^2$. Note that such an $m$ exists because $a,r$, and $s$ are coprime, and $n_1 n_2 n$ must be an odd square.

Now, let us estimate the sum over $m$ appearing in \eqref{equ3}. Using the definition of $W_2$ appearing in \eqref{eq:W2} we write
\begin{equation}\label{equint}
\begin{aligned}
    &\sum_{m \text{ odd}}\frac{\tau(rsm^2)}{m}\prod_{\substack{p \ \text{prime} \\ p\mid arsm}}\left(\frac{p}{p+1}\right)\int_1^2 W_2\left(\frac{2rsm^2\pi}{Xt}\right)dt \\ &=
\frac{1}{2\pi i}\int_{(c)}\left(\frac{\Gamma(\frac{w}{2}+\frac{1}{4})}{\Gamma(\frac{1}{4})}\right)^2\left(\frac{X}{2rs\pi}\right)^w\left(\int_{1}^{2}t^wdt\right) \left( \sum_{m\text{ odd}}\frac{\tau(rsm^2)}{m^{1+2w}}\prod_{\substack{p \ \text{prime} \\ p\mid arsm}}\left(\frac{p}{p+1}\right) \right) \frac{dw}{w},
\end{aligned}
\end{equation}
where $(c)$ is the vertical line going from $c-i\infty$ to $c+i\infty$.
Now, recall from \cite[Lemma 5.1]{Soundararajan_2000} that
\begin{equation}\label{Sound Lemma 5.1}
\sum_{m \text{ odd}} \frac{\tau(r s m^2)}{m^{1+2w}} \prod_{\substack{p \ \text{prime} \\p \mid a r s m}} \left(\frac{p}{p+1}\right) = \tau(rs) \zeta(1 + 2 w)^3 \eta(1+2w,a^2 r s),
\end{equation}
where $\eta$ is the function defined in \eqref{eq:eta}.
Therefore, we see that
\begin{equation}\label{equint2}
\begin{aligned}
    &\int_{(c)}\left(\frac{\Gamma(\frac{w}{2}+\frac{1}{4})}{\Gamma(\frac{1}{4})}\right)^2\left(\frac{X}{2rs\pi}\right)^w\left(\int_{1}^{2}t^wdt\right) \left( \sum_{m\text{ odd}}\frac{\tau(rsm^2)}{m^{1+2w}}\prod_{\substack{p \ \text{prime} \\ p\mid arsm}}\left(\frac{p}{p+1}\right) \right) \frac{dw}{w} \\
    = &\int_{(c)}\left(\frac{\Gamma(\frac{w}{2}+\frac{1}{4})}{\Gamma(\frac{1}{4})}\right)^2\left(\frac{X}{2rs\pi}\right)^w\left(\int_{1}^{2}t^wdt\right)\tau(rs)\zeta(1+2w)^3\eta(1+2w,a^2rs)\frac{dw}{w}.
\end{aligned}
\end{equation}
We can now deform the integration contour $(c)$ until we reach the pole $w = 0$ of the integrand, as done, for example, on \cite[Page 465]{Soundararajan_2000}. Doing so allows us to see that
\begin{equation} \label{eq:residue}
\begin{aligned}
    \frac{1}{2 \pi i} &\int_{(c)}\left(\frac{\Gamma(\frac{w}{2}+\frac{1}{4})}{\Gamma(\frac{1}{4})}\right)^2\left(\frac{X}{2rs\pi}\right)^w\left(\int_{1}^{2}t^wdt\right)\tau(rs)\zeta(1+2w)^3\eta(1+2w,a^2rs)\frac{dw}{w} \\ =
\tau(rs) &\Res_{w=0} \left( \left(\frac{\Gamma(\frac{w}{2}+\frac{1}{4})}{\Gamma(\frac{1}{4})}\right)^2\left(\frac{X}{2rs\pi}\right)^w\left(\int_{1}^{2}t^wdt\right)\frac{\zeta(1+2w)^3}{w}\eta(1+2w;a^2rs) \right) + O((rs)^{\frac{1}{4}}X^{-\frac{1}{4}+\varepsilon}).
\end{aligned}
\end{equation}
Let us now compute the residue appearing on the right-hand side of \eqref{eq:residue} in the following lemma.
\begin{lemma}\label{residues}
    Let $N$ and $X$ be large. Then
\begin{equation}
\label{eq:residue_0}
\begin{aligned}
&\Res_{w=0}\left( \left(\frac{\Gamma(\frac{w}{2}+\frac{1}{4})}{\Gamma(\frac{1}{4})}\right)^2\left(\frac{X}{2rs\pi}\right)^w\left(\int_{1}^{2}t^wdt\right)\frac{\zeta(1+2w)^3}{w}\eta(1+2w;a^2rs)\right)
\\ = &\frac{c_6}{48} \frac{rs}{\sigma(rs)h(ars)}\left( \left(\log\left(\frac{X}{rs}\right) \right)^3 +\CalQ(X,rsa^2)\right),
\end{aligned}
\end{equation}
where $\sigma$ is defined as in \eqref{eq:sigma}, $h$ is defined as in \eqref{eq:h}, while $c_6$ is defined as in \eqref{eq:c6} and $\mathcal{Q}(X,rsa^2)$ is defined as
\begin{equation*}
\CalQ(X,rsa^2)= Q_0\left(\log\left(\frac{X}{rs}\right)\right)
+\frac{\eta'}{\eta}
(1;a^2rs)Q_1\left(\log\left(\frac{X}{rs}\right)\right)
+ \frac{\eta''}{\eta}(1;a^2rs)Q_2\left(\log\left(\frac{X}{rs}\right)\right)+Q_3\frac{\eta'''}{\eta}(1;a^2rs),
\end{equation*}
where $Q_3$ is an absolute constant, while $Q_0$, $Q_1$, and $Q_2$ are absolute polynomials of degrees $2$, $2$, and $1$ respectively.
\end{lemma}
\begin{proof}
    We will proceed by computing the Laurent series expansions at $w = 0$ of each of the functions appearing on the left-hand side of \eqref{eq:residue_0}, following \cite[Section~5.1]{Soundararajan_2000}.
    First of all, it is easy to see that
    \[
        \begin{aligned}
            \left( \frac{X}{2 r s \pi} \right)^w &= \exp\left( w \log\left( \frac{X}{2 r s \pi} \right) \right) = \sum_{j=0}^{+\infty} \frac{1}{j!} \left( \log\left( \frac{X}{2 r s \pi} \right) \right)^j w^j \\
            \int_1^2 t^w dt &= \frac{2^{w+1}-1}{w+1} = \sum_{j=0}^{+\infty} (-1)^j \left( 1 + 2 \sum_{k=1}^j \frac{(-1)^k}{k!} (\log 2)^k \right) w^j \\
            \Gamma\left( \frac{w}{2} + \frac{1}{4} \right) &= \Gamma\left( \frac{1}{4}\right) + \sum_{j=1}^{+\infty} \frac{\Gamma^{(j)}\left( \frac{1}{4} \right)}{j! \cdot 2^j} w^j \\
            \eta(1+2w;a^2rs) &= \sum_{j=0}^{+\infty} \frac{2^j}{j!} \eta^{(j)}(1;a^2rs) w^j = \eta(1;a^2rs) \left(1 + \sum_{j=1}^{+\infty} \frac{2^j}{j!} \frac{\eta^{(j)}}{\eta}(1;a^2rs) w^j \right)
        \end{aligned}
    \]
    and it is known since the work of Briggs and Chowla \cite{Briggs_Chowla} that
    \[
        \zeta(1+2w) = \frac{1}{2 w} + \sum_{j=0}^{+\infty} (-2)^j \frac{\gamma_j}{j!} w^j
    \]
    where the coefficients $\{\gamma_j\}_{j=0}^{+\infty}$ are the Stieltjes constants, defined in \eqref{eq:Stieltjes}.

It follows that the residue appearing in \eqref{eq:residue_0} may be written as
\begin{align*}
\frac{\eta(1;a^2rs)}{48}\left(
\log^3\left(\frac{X}{2rs\pi}\right)+P_0\left(\log\left(\frac{X}{2rs\pi}\right)\right)\right.
+&\frac{\eta'}{\eta}
(1;a^2rs)P_1\left(\log\left(\frac{X}{2rs\pi}\right)\right)
\\
+&\left.\frac{\eta''}{\eta}(1;a^2rs)P_2\left(\log\left(\frac{X}{2rs\pi}\right)\right)+P_3\frac{\eta'''}{\eta}(1;a^2rs)\right),
\end{align*}
where $P_3$ is an absolute constant, and $P_0$, $P_1$, and $P_2$ are polynomials of degrees $2$, $2$, and $1$ respectively.
To conclude, it suffices to observe that 
\begin{align*}
    \eta(1;a^2rs) &= 2^{-3} \cdot \left( \prod_{\substack{p \ \text{prime} \\ p \geq 3 \\ p \mid rs}} \frac{p-1}{p + 1} \right) \cdot \left( \prod_{\substack{p \ \text{prime} \\ p \geq 3 \\ p \mid a}} \frac{p-1}{p} \right) \cdot \left( \prod_{\substack{p \ \text{prime} \\ p \geq 3 \\ p \nmid ars}} \frac{(p-1)(p^3+2p^2-2p+1)}{p^3(p+1)} \right) \\
    &= c_6 \cdot \left( \prod_{\substack{p \ \text{prime} \\ p \geq 3 \\ p \mid rs}} \frac{p^3}{p^3 + 2 p^2 - 2p + 1} \right) \cdot \left( \prod_{\substack{p \ \text{prime} \\ p \geq 3 \\ p \mid a}} \frac{p^2 (p+1)}{p^3+2p^2-2p+1} \right) \\ 
    &= c_6 \cdot rs \cdot \left( \prod_{\substack{p \ \text{prime} \\ p \geq 3 \\ p \mid rs}} \frac{1}{p+1} \right) \cdot \left( \prod_{\substack{p \ \text{prime} \\ p \geq 3 \\ p \mid ars}} \frac{p^2 (p+1)}{p^3 + 2p^2 - 2p + 1} \right)
    \\ &= c_6 \frac{rs}{\sigma(rs)h(ars)},
\end{align*}
and to use the identity $\log(\frac{X}{2rs\pi})=\log(\frac{X}{rs})-\log(2\pi)$ in order to obtain polynomials $Q_0,\dots,Q_3$ evaluated at $\log(\frac{X}{rs})$ rather than polynomials $P_0,\dots,P_3$ evaluated at $\log(\frac{X}{2rs\pi})$.
\end{proof}

Applying \cref{residues} to evaluate the residue appearing in \eqref{eq:residue}, we can evaluate asymptotically the sum appearing in \eqref{equ3}. More precisely, we have that
\begin{equation*}
\begin{aligned}
    \frac{X}{12\zeta(2)} &\sum_{\substack{a,r,s\\ar,as\leq N\\(a,r)=(a,s)=(r,s)=1}}
\mu(a)^2f(a)^2\frac{\mu(r)f(r)\mu(s)f(s)}{\sqrt{rs}} \sum_{m \text{ odd}}\frac{\tau(rsm^2)}{m}\prod_{\substack{p \ \text{prime} \\ p\mid arsm}}\left(\frac{p}{p+1}\right)\int_1^2 W_2\left(\frac{2rsm^2\pi}{Xt}\right)dt \\
= \frac{c_6X}{576\zeta(2)}&\sum_{\substack{a,r,s\\ar,as\leq N\\(a,r)=(a,s)=(r,s)=1}}
\mu(a)^2f(a)^2\frac{\mu(r)f(r)\mu(s)f(s)}{\sqrt{rs}}\frac{\tau(rs)rs}{\sigma(rs)h(ars)}\left(\log^3\left(\frac{X}{rs}\right)+\CalQ(X,rsa^2)\right) + O((rs)^{\frac{1}{4}}X^{-\frac{1}{4}+\varepsilon}).
\end{aligned}
\end{equation*}

We conclude that for $N\leq X^{\frac{1}{20}+\varepsilon}$ we have that
\begin{equation} \label{sum_truncated}
\begin{aligned}
    M_2(R,X) = &c_5 X\sum_{\substack{a,r,s\\ar,as\leq N\\(a,r)=(a,s)=(r,s)=1}}
\frac{\mu(a)^2f(a)^2}{h(a)}\frac{\mu(r)f(r)\tau(r)\sqrt{r}}{\sigma(r)h(r)}\frac{\mu(s)f(s)\tau(s)\sqrt{s}}{\sigma(s)h(s)} \left(\log^3\left(\frac{X}{rs}\right)+\CalQ(X,rsa^2)\right) \\ &+O(X^{1-\varepsilon}),
\end{aligned}
\end{equation}
where $c_5$ is the constant defined in \eqref{eq:c5}.

We now enter the final steps of the proof. In Step 0, we extend the summation using Rankin's trick. In Step 1, we obtain the asymptotic behavior of the sum involving $\log^3(\frac{X}{rs})$: this will be the main term. In Step 2, we show that the asymptotic behavior of the extended sum involving $\CalQ(X,rsa^2)$ is an error term. In Step 3, we show that the tails, coming from the extension in Step 0, are smaller that the error term. This leads to the conclusion in the Final Step.

\textbf{Step 0: extending the sums.} 
We extend the summations in \eqref{sum_truncated} over $a$, $r$, and $s$ to run over all integers. At the end of this proof, we will estimate the tails. This will help us find a cleaner expression in terms of infinite products. In other words, we consider the following extended summation:
\begin{equation} \label{eq:extended_sum}
\sum_{\substack{a,r,s \\
(a,r)=(a,s)=(r,s)=1}}
\frac{\mu(a)^2f(a)^2}{h(a)}\frac{\mu(r)f(r)\tau(r)\sqrt{r}}{\sigma(r)h(r)}\frac{\mu(s)f(s)\tau(s)\sqrt{s}}{\sigma(s)h(s)} \left(\log^3\left(\frac{X}{rs}\right)+\CalQ(X,rsa^2)\right)
\end{equation}
and we will first focus on the first term
\begin{equation}\label{last_sum}
\sum_{\substack{a,r,s\\(a,r)=(a,s)=(r,s)=1}}
\frac{\mu(a)^2f(a)^2}{h(a)}\frac{\mu(r)f(r)\tau(r)\sqrt{r}}{\sigma(r)h(r)}\frac{\mu(s)f(s)\tau(s)\sqrt{s}}{\sigma(s)h(s)}\left(\log X-\log(rs)\right)^3,
\end{equation}
which will be the main term, as we will show in the following step.

\textbf{Step 1: the main term.}
We want to determine the asymptotic behavior of \eqref{last_sum}. 
To do so, we first of all notice that
\[
\begin{aligned}
    \sum_{\substack{a,r,s\\(a,r)=(a,s)=(r,s)=1}}
\frac{\mu(a)^2f(a)^2}{h(a)}\frac{\mu(r)f(r)\tau(r)\sqrt{r}}{\sigma(r)h(r)}\frac{\mu(s)f(s)\tau(s)\sqrt{s}}{\sigma(s)h(s)} &= \sum_{m=1}^{+\infty} \frac{\mu(m) f(m) \tau(m)^2 \sqrt{m}}{\sigma(m) h(m)} \left( \sum_{a \mid m} \frac{\sigma(a) \mu(a) f(a)}{\tau(a)^2 \sqrt{a}} \right) \\
&=
\prod_{p\text{ prime}}F(p)
\end{aligned}
\]
where $F$ is the multiplicative function defined in \eqref{eq:F}.
More generally, for every $t \in \mathbb{N}$ we have that
\begin{equation} \label{eq:multiplicative_1}
\sum_{\substack{a,r,s\\(a,r)=(a,s)=(r,s)=1}}
\frac{\mu(a)^2f(a)^2}{h(a)}\frac{\mu(r)f(r)\tau(r)\sqrt{r}}{\sigma(r)h(r)}\frac{\mu(s)f(s)\tau(s)\sqrt{s}}{\sigma(s)h(s)}\log^t(rs) = \left(\sum_{\ell_{1},\dots,\ell_{t}\text{ primes}}H(\ell_{1}\cdots\ell_{t})\right) \prod_{p\text{ prime}}F(p),
\end{equation}
where $H$ is the multiplicative function defined in \eqref{eq:H}.
To see why \eqref{eq:multiplicative_1} holds true, we start by factoring the square-free number $rs$ into a product of primes, which shows that
\begin{equation} \label{eq:multiplicative_2}
    \begin{aligned} &\sum_{\substack{a,r,s\\(a,r)=(a,s)=(r,s)=1}}
\frac{\mu(a)^2f(a)^2}{h(a)}\frac{\mu(r)f(r)\tau(r)\sqrt{r}}{\sigma(r)h(r)}\frac{\mu(s)f(s)\tau(s)\sqrt{s}}{\sigma(s)h(s)}\log^t(rs) \\ = &\sum_{\ell_{1},\dots,\ell_{t}\text{ primes}}  \left( \sum_{\substack{a,r,s \\ \substack{(a,r)=(a,s)=(r,s)=1}\\
\ell_i\mid rs}}
\frac{\mu(a)^2f(a)^2}{h(a)}\frac{\mu(r)f(r)\tau(r)\sqrt{r}}{\sigma(r)h(r)}\frac{\mu(s)f(s)\tau(s)\sqrt{s}}{\sigma(s)h(s)} \right)\prod_{i=1}^t \log\ell_{i}
\end{aligned}
\end{equation}
for every $t \in \mathbb{N}$. Then, we observe that
\begin{equation} \label{eq:multiplicative_3}
\sum_{\substack{a,r,s \\ \substack{(a,r)=(a,s)=(r,s)=1}\\
\ell_i\mid rs}}
\frac{\mu(a)^2f(a)^2}{h(a)}\frac{\mu(r)f(r)\tau(r)\sqrt{r}}{\sigma(r)h(r)}\frac{\mu(s)f(s)\tau(s)\sqrt{s}}{\sigma(s)h(s)}
= 
\prod_{\substack{p\text{ prime}\\p\nmid \ell_{1}\cdots\ell_{t} }}F(p) \prod_{\substack{p\text{ prime}\\p\mid \ell_{1}\cdots\ell_{t} }}\left(-2\frac{\tau(p) f(p) \sqrt{p}}{h(p)\sigma(p)}\right)
\end{equation}
as follows from the definition of $F$.
Finally, we use the definition of $H$ to write
\begin{align*}
&\sum_{\ell_{1},\dots,\ell_{t}\text{ primes}} \left(\prod_{\substack{p\text{ prime}\\p\nmid \ell_{1}\cdots\ell_{t} }}F(p) \prod_{\substack{p\text{ prime}\\p\mid \ell_{1}\cdots\ell_{t} }}\left(-2\frac{\tau(p) f(p) \sqrt{p}}{h(p)\sigma(p)}\right)\right) \prod_{i=1}^t \log\ell_{i} \\ = 
&\sum_{\ell_{1},\dots,\ell_{t}\text{ primes}} \left(\prod_{\substack{p\text{ prime}\\p\nmid \ell_{1}\cdots\ell_{t} }} F(p) \right) \cdot \left(  \prod_{\substack{p\text{ prime}\\p\mid \ell_{1}\cdots\ell_{t} }}\left(-2\frac{\tau(p) f(p) \sqrt{p}}{h(p)\sigma(p)}\right)\right)
\left(\prod_{i=1}^t \log\ell_{i}\right) \\
= &\sum_{\ell_{1},\dots,\ell_{t}\text{ primes}}\left(\prod_{\substack{p\text{ prime}\\p\nmid \ell_{1}\cdots\ell_{t} }}F(p)\right)
\left(H(\ell_{1}\cdots\ell_{t})\prod_{\substack{p\text{ prime}\\p\mid \ell_{1}\cdots\ell_{t} }}F(p)\right) \\ =
&\left(\sum_{\ell_{1},\dots,\ell_{t}\text{ primes}}H(\ell_{1}\cdots\ell_{t})\right) \prod_{p\text{ prime}}F(p)
\end{align*}
and we observe that these equalities, together with \eqref{eq:multiplicative_2} and \eqref{eq:multiplicative_3}, imply \eqref{eq:multiplicative_1}.

These formulas allow us to rewrite the sum \eqref{last_sum} in multiplicative form. More precisely, we have that
\begin{equation}\label{productexpression}
\begin{aligned}
&\sum_{\substack{a,r,s\\(a,r)=(a,s)=(r,s)=1}}
\frac{\mu(a)^2f(a)^2}{h(a)}\frac{\mu(r)f(r)\tau(r)\sqrt{r}}{\sigma(r)h(r)}\frac{\mu(s)f(s)\tau(s)\sqrt{s}}{\sigma(s)h(s)}\left(\log X-\log(rs)\right)^3 = \\ &\left(\log^3X-3\log^2 X\sum_{\ell\text{ prime}}H(\ell)+3\log X\sum_{\ell_{1},\ell_{2}\text{ primes}}H(\ell_{1}\ell_{2})-\sum_{\ell_{1},\ell_{2},\ell_{3}\text{ primes}}H(\ell_{1}\ell_{2}\ell_{3})\right)\prod_{p\ \text{prime}}F(p).
\end{aligned}
\end{equation}

We can finally use this multiplicative form to obtain the asymptotic behavior of \eqref{last_sum}. To do so, note that
\begin{equation}\label{estimates H}
\sum_{\ell_{1},\dots,\ell_{t}\text{ primes}}|H(\ell_{1}\cdots\ell_{t})|\ll \log^{\frac{t}{2}+\varepsilon}X,    
\end{equation}
for every $t \in \mathbb{N}$,
as follows from the prime number theorem and the definition of $H$ given in \eqref{eq:H}. In particular, recall that in our case $L = \sqrt{\log N \log\log N}$ and $N \leq X^{\frac{1}{20} - \varepsilon}$.
Finally, we can observe that 
\[
    \prod_{p \ \text{prime}} F(p) = \prod_{p \ \text{prime}} \left( 1+\frac{f(p)^2}{h(p)} - 4 \frac{f(p) \sqrt{p}}{h(p) (p+1)} \right) \sim \prod_{p \ \text{prime}} \left( 1 + f(p)^2 - 4 \frac{f(p)}{\sqrt{p}} \right),
\]
as follows from the definition of $F$, given in \eqref{eq:F}, and the definition of $h$, given in \eqref{eq:h}. These identities allow us to see that, when $X$ is large, we have that
\begin{equation}\label{eq:Asymptotic}
\begin{aligned}
\sum_{\substack{a,r,s\\(a,r)=(a,s)=(r,s)=1}}
\frac{\mu(a)^2f(a)^2}{h(a)}\frac{\mu(r)f(r)\tau(r)\sqrt{r}}{\sigma(r)h(r)}\frac{\mu(s)f(s)\tau(s)\sqrt{s}}{\sigma(s)h(s)} \log^3\left( \frac{X}{rs} \right) \sim (\log^3 X)\prod_{p \ \text{prime}}\left(1+f(p)^2-4\frac{f(p)}{\sqrt{p}}\right).
\end{aligned}
\end{equation}
We keep the estimate \eqref{eq:Asymptotic} in our records and move to the next step.

\textbf{Step 2: the error term.}
Let us now prove that the second term in the extended sum \eqref{eq:extended_sum}, which is given by
\begin{align}\label{Qsum}
\sum_{\substack{a,r,s \\ (a,r)=(a,s)=(r,s)=1}}
&\frac{\mu(a)^2f(a)^2}{h(a)}\frac{\mu(r)f(r)\tau(r)\sqrt{r}}{\sigma(r)h(r)}\frac{\mu(s)f(s)\tau(s)\sqrt{s}}{\sigma(s)h(s)} \CalQ(X,a^2rs) 
\end{align}
is little-$o$ of \eqref{eq:Asymptotic}.
To do so, let us recall from \cref{residues} that $\mathcal{Q}(X,a^2rs)$ is defined as
\begin{equation*}
\CalQ(X,a^2rs)= Q_0\left(\log\left(\frac{X}{rs}\right)\right)
+\frac{\eta'}{\eta}
(1;a^2rs)Q_1\left(\log\left(\frac{X}{rs}\right)\right)
+ \frac{\eta''}{\eta}(1;a^2rs)Q_2\left(\log\left(\frac{X}{rs}\right)\right)+Q_3\frac{\eta'''}{\eta}(1;a^2rs),
\end{equation*}
where $Q_3$ is an absolute constant, while $Q_0$, $Q_1$, and $Q_2$ are absolute polynomials of degrees $2$, $2$, and $1$ respectively.

The same argument used in Step 1 can be used to prove that there exists an absolute constant $c_7$ such that
\begin{equation} \label{eq:Q0_estimate}
    \begin{aligned}
        &\sum_{\substack{a,r,s\\(a,r)=(a,s)=(r,s)=1}}
\frac{\mu(a)^2f(a)^2}{h(a)}\frac{\mu(r)f(r)\tau(r)\sqrt{r}}{\sigma(r)h(r)}\frac{\mu(s)f(s)\tau(s)\sqrt{s}}{\sigma(s)h(s)}Q_0\left(\log\left(\frac{X}{rs}\right)\right) \\ \sim &c_7 (\log^2 X)\prod_{p \ \text{prime}}\left(1+f(p)^2-4\frac{f(p)}{\sqrt{p}}\right).
    \end{aligned}
\end{equation}
Then, to prove that \eqref{Qsum} is little-o of \eqref{eq:Asymptotic}, it is enough to show that the sum 
\begin{align}\label{Q1sum}
\sum_{\substack{a,r,s \\(a,r)=(a,s)=(r,s)=1}}
&\frac{\mu(a)^2f(a)^2}{h(a)}\frac{\mu(r)f(r)\tau(r)\sqrt{r}}{\sigma(r)h(r)}\frac{\mu(s)f(s)\tau(s)\sqrt{s}}{\sigma(s)h(s)}\CalQ_1(X,a^2rs) 
\end{align}
is little-$o$ of \eqref{eq:Asymptotic}, where $\CalQ_1(X,a^2rs)\coloneq \CalQ(X,a^2rs)-Q_0(\log(\frac{X}{rs}))$.

Notice that for every $\alpha\in\C$ with $\text{Re}(\alpha)>\frac{1}{2}$, and every $l \in \mathbb{N}$ we can rewrite $\eta(\alpha;l)$ as
\[
\eta(\alpha;l)=\eta(\alpha;1)\prod_{\substack{p \ \text{prime} \\ p\mid l}}\frac{\eta_p(\alpha;l)}{\eta_p(\alpha;1)} = \eta(\alpha;1) \prod_{\substack{p \ \text{prime} \\ p \mid l}} G_p(\alpha;l),
\]
where $G_p(\alpha;l)\coloneqq \frac{\eta_p(\alpha;l)}{\eta_p(\alpha;1)}$, as defined in \eqref{eq:Gp}. 
This allows us to rewrite $\CalQ_1(X,a^2rs)$ as follows:
\begin{equation}\label{Q from G}
\begin{split}
\mathcal{Q}_1(X,a^2rs) &= Q_1\left(\log\left(\frac{X}{rs}\right)\right) \cdot \frac{\eta'}{\eta}
(1;a^2rs)
+ Q_2\left(\log\left(\frac{X}{rs}\right)\right) \cdot \frac{\eta''}{\eta}(1;a^2rs) + Q_3 \cdot \frac{\eta'''}{\eta}(1;a^2rs) \\
&= Q_1\left(\log\left(\frac{X}{rs}\right)\right) \left(\frac{\eta'(1;1)}{\eta(1;1)} + \sum_{\substack{p \ \text{prime} \\ p\mid l}}\frac{G'_{p}(1;l)}{G_{p}(1;l)}\right) \\ &+ Q_2\left(\log\left(\frac{X}{rs}\right)\right)
\left( \frac{\eta''(1;1)}{\eta(1;1)} + 2 \frac{\eta'(1;1)}{\eta(1;1)} \sum_{\substack{p \ \text{prime} \\ p \mid l}} \frac{G_p'(1;l)}{G_p(1;l)} + \sum_{\substack{p,q \ \text{primes} \\ p,q\mid l\\p\neq q}} \frac{G'_{p}(1;l)}{G_{p}(1;l)}\frac{G'_{q}(1;l)}{G_{q}(1;l)}
+ \sum_{\substack{p \ \text{prime} \\ p\mid l}} \frac{G''_{p}(1;l)}{G_{p}(1;l)}\right)
\\
&+ Q_3 \cdot \left( \frac{\eta'''(1;1)}{\eta(1;1)} + 3 \frac{\eta''(1;1)}{\eta(1;1)} \sum_{\substack{p \ \text{prime} \\ p \mid \ell}} \frac{G_p'(1;l)}{G_p(1;l)} + 3 \frac{\eta'(1;1)}{\eta(1;1)} \left( \sum_{\substack{p,q \ \text{primes} \\ p,q\mid l\\p\neq q}} \frac{G'_{p}(1;l)}{G_{p}(1;l)}\frac{G'_{q}(1;l)}{G_{q}(1;l)}
+ \sum_{\substack{p \ \text{prime} \\ p\mid l}} \frac{G''_{p}(1;l)}{G_{p}(1;l)} \right) \right. 
\\ &+ \left. \sum_{\substack{\ell, p, q \ \text{primes} \\ \ell,p,q\mid l\\ \ell\neq p\neq q}} \frac{G'_{\ell}(1;l)}{G_{\ell}(1;l)}\frac{G'_{p}(1;l)}{G_{p}(1;l)}\frac{G'_{q}(1;l)}{G_{q}(1;l)}+
\sum_{\substack{p,q \ \text{primes} \\ p,q\mid l\\p\neq q}} \frac{G''_{p}(1;l)}{G_{p}(1;l)}\frac{G'_{q}(1;l)}{G_{q}(1;l)}
+ \sum_{\substack{p \ \text{prime} \\ p\mid l}} \frac{G'''_{p}(1;l)}{G_{p}(1;l)}\right)
\end{split}
\end{equation}
where $l=a^2rs$. 

In order to use this explicit expression to estimate the asymptotic behavior of \eqref{Q1sum}, we need to know the asymptotic behavior of 
\[
    \frac{G_p^{(i)}(1;a^2rs)}{G_p(1;a^2 r s)}
\]
whenever $p$ is a large prime and $i \in \{1,2,3\}$. To establish this behavior, we use the following explicit expressions
\begin{align}
\label{eq:G_logarithmic_derivative_1}\frac{G'_{p}(\alpha;l)}{G_{p}(\alpha;l)}&= \frac{\eta'_p(\alpha;l)}{\eta_p(\alpha;l)}-\frac{\eta'_p(\alpha;1)}{\eta_p(\alpha;1)},\\
\frac{G''_{p}(\alpha;l)}{G_{p}(\alpha;l)}&=\frac{\eta''_p(\alpha;l)}{\eta_p(\alpha;l)}-2\frac{\eta'_p(\alpha;l)}{\eta_p(\alpha;l)}\frac{\eta'_p(\alpha;1)}{\eta_p(\alpha;1)}+2\left(\frac{\eta'_p(\alpha;1)}{\eta_p(\alpha;1)}\right)^2-\frac{\eta''_p(\alpha;1)}{\eta_p(\alpha;1)}, \label{eq:G_logarithmic_derivative_2}\\
\label{eq:G_logarithmic_derivative_3} \frac{G'''_{p}(\alpha;l)}{G_{p}(\alpha;l)}&=\frac{\eta'''_p(\alpha;l)}{\eta_p(\alpha;l)} -3 \frac{\eta_p''(\alpha;l) \eta_p'(\alpha;1) + \eta_p'(\alpha;l) \eta_p''(\alpha;1)}{\eta_p(\alpha;l) \eta_p(\alpha;1)} + 6 \frac{\eta_p'(\alpha;l)}{\eta_p(\alpha;l)} \left( \frac{\eta_p'(\alpha;1)}{\eta_p(\alpha;1)} \right)^2 \\ \notag &- \frac{\eta_p'''(\alpha;1)}{\eta_p(\alpha;1)} - 6 \left( \frac{\eta_p'(\alpha;1)}{\eta_p(\alpha;1)} \right)^3 + 6 \frac{\eta_p'(\alpha;1) \cdot \eta_p''(\alpha;1)}{(\eta_p(\alpha;1))^2}
\end{align}
which follow from elementary calculus.
Now, we see from the definition of $\eta_p(\alpha;l)$, given in \eqref{eq:eta}, that
\[
    \eta_p^{(i)}(\alpha;l) = \begin{cases}
        (-1)^{i+1} \frac{\log^{i}p}{(p+1)p^\alpha} \left( 3 + 2^i (p-3) p^{-\alpha} + 3^i p^{-2\alpha} \right), \ \text{if} \ p \nmid l \\
        (-1)^{i+1} \left( \frac{p}{p+1} \right) \frac{\log^{i}p}{p^\alpha}, \ \text{if} \ p \mid l_1 \\
        (-1)^{i+1} 2^i \left( \frac{p}{p+1} \right) \frac{\log^{i}p}{p^{2 \alpha}}, \ \text{if} \ p \mid l_2
    \end{cases}
\]
for every $i \geq 1$ and every $l = l_1 l_2^2$, where $l_1$ is squarefree and coprime to $l_2$. Therefore, we have that
\[
\frac{\eta_p^{(i)}(1;l)}{\eta_p(1,l)} = \begin{cases}
    (-1)^{i+1} (2^i + 3) \frac{\log^ip}{p^2} \left( 1 + O(\frac{1}{p}) \right), \ &\text{if} \ p \nmid l \\
    (-1)^{i+1} \frac{\log^ip}{p} \left( 1 + O(\frac{1}{p}) \right), \ &\text{if} \ p \mid l_1 \\
    (-1)^{i+1} 2^i \frac{\log^ip}{p^2} \left( 1 + O(\frac{1}{p}) \right), \ &\text{if} \ p \mid l_2
\end{cases}
\]
as $p \to +\infty$. Using now the explicit formulas given in \eqref{eq:G_logarithmic_derivative_1}, \eqref{eq:G_logarithmic_derivative_2} and \eqref{eq:G_logarithmic_derivative_3}, we see that
\begin{equation} \label{eq:Gp_asymptotics_main}
    \frac{G_p^{(i)}(1;l)}{G_p(1;l)} = \begin{cases}
        c_2(i) \frac{\log^ip}{p} \mathcal{H}_{1,i}(p), \ &\text{if} \ p \mid l_1, \\
        c_3(i) \frac{\log^ip}{p^2} \mathcal{H}_{2,i}(p), \ &\text{if} \ p \mid l_2,
    \end{cases}
\end{equation}
for some constants $c_2(i)$ and $c_3(i)$, depending only on $i$, and some explicit functions $\mathcal{H}_{1,i}(p), \mathcal{H}_{2,i}(p)$ such that \[\mathcal{H}_{1,i}(p), \mathcal{H}_{2,i}(p) = 1 + O\left( \frac{1}{p} \right)\] 
as $p \to +\infty$. This explicit asymptotic behavior allows us to determine that the sum appearing in \eqref{Q1sum} is little-$o$ of \eqref{eq:Asymptotic}. More precisely, we have that
\begin{equation} \label{eq:Q1_estimate}
    \begin{aligned}
        &\phantom{=} \sum_{\substack{a,r,s \\(a,r)=(a,s)=(r,s)=1}}
\frac{\mu(a)^2f(a)^2}{h(a)}\frac{\mu(r)f(r)\tau(r)\sqrt{r}}{\sigma(r)h(r)}\frac{\mu(s)f(s)\tau(s)\sqrt{s}}{\sigma(s)h(s)}\CalQ_1(X,a^2rs) \\ &= 
O\left(\log^{2} X \prod_{p \ \text{prime}}\left(1+f(p)^2-4\frac{f(p)}{\sqrt{p}}\right)\right)
    \end{aligned}
\end{equation}
as $X \to +\infty$.
To obtain this, one can study separately each summand of \eqref{Q from G}, following the method used in Step 1.

For the sake of clarity, we outline the procedure only for the following sum 
\begin{align}\label{Q1sumand}
\sum_{\substack{a,r,s\\(a,r)=(a,s)=(r,s)=1}}
&\frac{\mu(a)^2f(a)^2}{h(a)}\frac{\mu(r)f(r)\tau(r)\sqrt{r}}{\sigma(r)h(r)}\frac{\mu(s)f(s)\tau(s)\sqrt{s}}{\sigma(s)h(s)}Q_1\left(\log\left(\frac{X}{rs}\right)\right)\left(\sum_{\substack{p \ \text{prime} \\ p\mid a^2rs}}\frac{G'_{p}(1;a^2rs)}{G_{p}(1;a^2rs)}\right).
\end{align}
First of all, applying \eqref{eq:Gp_asymptotics_main} with $l = a^2 r s$, and recalling that $Q_1$ is a polynomial of degree two we notice that \eqref{Q1sumand} is a linear combination of
\begin{equation}\label{Q1 error}
    \log^{j}X\sum_{\substack{a,r,s\\(a,r)=(a,s)=(r,s)=1}} \left( \begin{aligned}
&\frac{\mu(a)^2f(a)^2}{h(a)}\frac{\mu(r)f(r)\tau(r)\sqrt{r}}{\sigma(r)h(r)}\frac{\mu(s)f(s)\tau(s)\sqrt{s}}{\sigma(s)h(s)} \\ \cdot &\log^{i}(rs)\left(c_2(1) \sum_{\substack{p \ \text{prime} \\ p\mid rs}}\frac{\log p}{p}\CalH_{1,1}(p)+c_3(1)\sum_{\substack{p \ \text{prime} \\ p\mid a}}\frac{\log p}{p^2}\CalH_{2,1}(p)\right) \end{aligned} \right)
\end{equation}
for $0\leq j+i\leq 2$. By expanding $\log^{i}(rs)$ we see that
\begin{align*}
&\log^{i}(rs)\left(c_2(1) \sum_{\substack{p \ \text{prime} \\ p\mid rs}}\frac{\log p}{p}\CalH_{1,1}(p)+c_3(1) \sum_{\substack{p \ \text{prime} \\ p\mid a}}\frac{\log p}{p^2}\CalH_{2,1}(p)\right) \\ = \ &c_2(1) \sum_{\substack{p,q_1,\dots,q_i \ \text{primes} \\ p,q_1,\dots,q_i\mid rs}} \frac{\log p}{p} \left( \prod_{k=1}^i \log q_k \right) \CalH_{1,1}(p)+ c_3(1) \sum_{\substack{p,q_1,\dots,q_i \ \text{primes} \\ p\mid a, \ q_1,\dots,q_i\mid rs}}\frac{\log p}{p^2} \left( \prod_{k=1}^i \log q_k \right) \CalH_{2,1}(p).
\end{align*}
Analogously to the procedure in Step 1, we can use this expression to write \eqref{Q1 error} as an Euler product. More precisely, \eqref{Q1 error} equals to 
\begin{align}\label{product formula Q1}
\log^{j}X \left( \prod_{p\text{ prime}}F(p) \right) \left(c_2(1)\sum_{\ell_{1},\dots,\ell_{i+1}
\text{ primes}}M(\ell_{1}\cdots\ell_{i+1})+c_3(i)\sum_{\substack{q,\ell_{1},\dots,\ell_{i}
\text{ primes}\\q \nmid \ell_{1}\cdots\ell_{i}}}G(q)H(\ell_{1}\cdots\ell_{i})\right),
\end{align}
where $F,G,H$ are the multiplicative functions defined in \eqref{eq:F}, \eqref{eq:G} and \eqref{eq:H}, while $M$ is the function defined in \eqref{eq:M}.
Since $p\geq L^2$ we have that
\[
|M(\ell_{1}\cdots\ell_{i+1})|\ll \frac{1}{L^2}|H(\ell_{1}\cdots\ell_{i+1})|
\]
and consequently by \eqref{estimates H} we have that
\[
\left\lvert \sum_{\ell_{1},\dots,\ell_{i+1}
\text{ primes}}M(\ell_{1}\cdots\ell_{i+1})\right\rvert \leq \frac{1}{L^2} \sum_{\ell_{1},\dots,\ell_{i+1}
\text{ primes}} \lvert H(\ell_{1}\cdots\ell_{i+1}) \rvert \ll \frac{1}{L^2}\log^{\frac{i+1}{2}+\varepsilon}X.
\]
On the other hand, we have that
\[
\sum_{p\text{ prime}}G(p)\ll L^2 \sum_{\substack{p \ \text{prime} \\ L^ 2\leq p\leq \exp(\log^2(L))}} \frac{1}{p^3\log p}\leq \frac{1}{L^2},
\]
where the last inequality can be proven using the prime number theorem, as done for the last equality appearing in \eqref{eq:main_theorem_from_moments}.
Finally, we notice that
\[
\left\lvert \sum_{\substack{p,\ell_{1},\dots,\ell_{i}
\text{ primes} \\ p\nmid \ell_{1}\cdots\ell_{i}}}G(p)H(\ell_{1}\cdots\ell_{i})\right\rvert
\leq \left( \sum_{p \ \text{prime}} G(p) \right) \left( \sum_{\ell_{1},\dots,\ell_{i}
\text{ primes }} \lvert H(\ell_{1}\cdots\ell_{i}) \rvert \right) \leq
\frac{1}{L^2}\log^{\frac{i}{2}+\varepsilon}X,
\]
thanks to the Cauchy-Schwartz inequality and to the previous bounds.
In particular, \eqref{product formula Q1} and consequently \eqref{Q1sum} are
\[
O\left(\log^{\frac{5}{2}+\varepsilon} X\prod_{p \ \text{prime}}\left(1+f(p)^2-4\frac{f(p)}{\sqrt{p}}\right)\right),
\]
for every $\varepsilon>0$, which is little-o of \eqref{eq:Asymptotic} as desired.

\textbf{Step 3: the tails.} We will now justify that we may extend the summation range in the right-hand side of \eqref{sum_truncated}, which is inspired by Rankin's trick as in \cite[Page~472]{Soundararajan_2008}, guaranteeing that the estimates in Step 1 and Step 2 are useful.
To do so, we need to prove that
\begin{align}
    \sum_{\substack{a,r,s \\
(a,r)=(a,s)=(r,s)=1 \\ (ar > N) \vee (as > N)}}
\frac{\mu(a)^2f(a)^2}{h(a)}\frac{\mu(r)f(r)\tau(r)\sqrt{r}}{\sigma(r)h(r)}\frac{\mu(s)f(s)\tau(s)\sqrt{s}}{\sigma(s)h(s)} \log^3\left(\frac{X}{rs}\right) &= o\left( (\log^3 X)\prod_{p \ \text{prime}}\left(1+f(p)^2-4\frac{f(p)}{\sqrt{p}}\right) \right) \label{eq:tail_main} \\
\sum_{\substack{a,r,s \\
(a,r)=(a,s)=(r,s)=1 \\ (ar > N) \vee (as > N)}}
\frac{\mu(a)^2f(a)^2}{h(a)}\frac{\mu(r)f(r)\tau(r)\sqrt{r}}{\sigma(r)h(r)}\frac{\mu(s)f(s)\tau(s)\sqrt{s}}{\sigma(s)h(s)} \mathcal{Q}(X,rsa^2) &= o\left( (\log^3 X)\prod_{p \ \text{prime}}\left(1+f(p)^2-4\frac{f(p)}{\sqrt{p}}\right) \right) \label{eq:tail_error}
\end{align}
for $X \to +\infty$ and $N = X^{\frac{1}{20} - \varepsilon}$.

Let us start by proving \eqref{eq:tail_main}. To do so, let $\alpha := \frac{1}{\log^3(L)}$ and notice that 
\begin{equation} \label{eq:rankin_elementary_inequality}
    (\log X+\log(rs))^3\ll (\log^3X)N^{-\alpha}(ars)^{\alpha}    
\end{equation}
for every $a,r,s$ such that $ars > N$. Using this inequality, we can check that the absolute value of the left hand side of {\eqref{eq:tail_main}} is bounded from above by an explicit Euler product, which can then be compared to the right hand side of {\eqref{eq:tail_main}}. More precisely, we have that:
\begin{equation} \label{eq:rankin_first_1}
\begin{aligned}
&\phantom{\leq} \left\lvert \sum_{\substack{a,r,s \\
(a,r)=(a,s)=(r,s)=1 \\ (ar > N) \vee (as > N)}}
\frac{\mu(a)^2f(a)^2}{h(a)} \, \frac{\mu(r)f(r)\tau(r)\sqrt{r}}{\sigma(r)h(r)} \, \frac{\mu(s)f(s)\tau(s)\sqrt{s}}{\sigma(s)h(s)} \log^3\left(\frac{X}{rs}\right) \right\rvert \\
&\leq \sum_{\substack{a,r,s \\
(a,r)=(a,s)=(r,s)=1 \\ (ar > N) \vee (as > N)}} \frac{f(a)^2}{h(a)} \, \frac{f(r) \tau(r)}{\sigma(r) \frac{h(r)}{\sqrt{r}}} \, \frac{f(s) \tau(s)}{\sigma(s) \frac{h(s)}{\sqrt{s}}} \, (\log X + \log(rs))^3 \\
&\ll (\log^3X) N^{-\alpha} \sum_{\substack{a,r,s \\
(a,r)=(a,s)=(r,s)=1 \\ (ar > N) \vee (as > N)}} f(a)^2 a^\alpha \, \frac{f(r) \tau(r) f(s) \tau(s)}{\sqrt{rs}} \, (rs)^\alpha,
\end{aligned}
\end{equation}
where the last step uses \eqref{eq:rankin_elementary_inequality} together with the fact that $\sigma(r) \geq r$ and $h(r) \geq 1$ for every $r \in \mathbb{N}$.
Moreover, note that
\begin{equation} \label{eq:rankin_first_2}
    \begin{aligned}
        \sum_{\substack{a,r,s \\
(a,r)=(a,s)=(r,s)=1 \\ (ar > N) \vee (as > N)}} f(a)^2 a^\alpha \, \frac{f(r) \tau(r) f(s) \tau(s)}{\sqrt{rs}} \, (rs)^\alpha &\leq \sum_{\substack{a,r,s \\
(a,r)=(a,s)=(r,s)=1}} f(a)^2 a^\alpha \, \frac{f(r) \tau(r) f(s) \tau(s)}{\sqrt{rs}} \, (rs)^\alpha \\
&= \prod_{p \ \text{prime}}\left(1+f(p)^2p^{\alpha}-4f(p)p^{\alpha-\frac{1}{2}}\right).
    \end{aligned}
\end{equation}
Combining \eqref{eq:rankin_first_1} with \eqref{eq:rankin_first_2}, we see that
\begin{align*}
\frac{\displaystyle{\left\lvert \sum_{\substack{a,r,s \\
(a,r)=(a,s)=(r,s)=1 \\ (ar > N) \vee (as > N)}}
\frac{\mu(a)^2f(a)^2}{h(a)} \, \frac{\mu(r)f(r)\tau(r)\sqrt{r}}{\sigma(r)h(r)} \, \frac{\mu(s)f(s)\tau(s)\sqrt{s}}{\sigma(s)h(s)} \log^3\left(\frac{X}{rs}\right) \right\rvert}}{\displaystyle{(\log^3 X)\prod_{p \ \text{prime}}\left(1+f(p)^2-4\frac{f(p)}{\sqrt{p}}\right)}}
&\ll N^{-\alpha} \prod_{p \ \text{prime}} \left( \frac{1+f(p)^2p^{\alpha}-4f(p)p^{\alpha-\frac{1}{2}}}{1+f(p)^2-4 f(p) p^{-\frac{1}{2}}} \right)\\
&\ll \exp\left(-\alpha\frac{\log N}{(\log\log N)^2}\right),
\end{align*}
where the last inequality is proven using the prime number theorem, similarly to how we proved the last equality appearing in \eqref{eq:main_theorem_from_moments}. The aforementioned inequalities clearly imply \eqref{eq:tail_main}. 
Finally, \eqref{eq:tail_error} is proven analogously, by expanding the polynomial $\mathcal{Q}$ and taking into account that for $\alpha=\frac{1}{\log^3L}$ and for any integers $n$ and $m$ we have that $\log^nX\log^m(rs)\ll (\log^{n+m}X)N^{-\alpha}(ars)^{\alpha}$.

\textbf{Final step}. Combining \eqref{sum_truncated} with the asymptotic estimate for the main term, provided by \eqref{eq:Asymptotic}, and the asymptotic estimates for the error term, provided by \eqref{eq:Q0_estimate} and \eqref{eq:Q1_estimate}, as well as the asymptotic estimates of the tails, provided by \eqref{eq:tail_main} and \eqref{eq:tail_error}, we obtain that
\[
M_2(R,X)\sim c_5 X(\log^3 X)\prod_{p \ \text{prime}}\left(1+f(p)^2-4\frac{f(p)}{\sqrt{p}}\right),
\]
and we finish the proof of \cref{SecondMoment}.
\end{proof}


\section{The right half of the critical strip}\label{sec:cricialright}

The aim of this section is to show that the ``Dedekind height'' $h_\sigma \colon \mathcal{N} \to \mathbb{R}$ introduced in \eqref{eq:Dedekind_height}, does not have the Bogomolov property if $\frac{1}{2} < \sigma \leq 1$. This was done by Généreux and Lalín in \cite{GeLa24}, conditionally on the validity of the generalized Riemann hypothesis. 

We will mention two ways in which one can obtain this result unconditionally. The first is obtained by adapting the resonator method to the right half of the critical strip, as suggested by Soundararajan, while the second is based on a result of Lamzouri \cite{Lamzouri_2011} for $\frac{1}{2} < \sigma < 1$, and a result of Granville and Soundararajan \cite{Granville_Soundararajan_2003} for $\sigma = 1$.

\subsection{Resonator method} In this section, we show that the resonator method used in the previous \cref{sec:cricial_center} can also be adapted to the right half of the critical strip. To do so, let us first define 
    \begin{equation}\label{resonator2}
     R(8d) := \sum_{l \leq X} r(l) \chi_{8d}(l)   
    \end{equation}
     analogously to \cref{sec:cricial_center}. Here we take once again $r(l) = \mu(l) f(l)$, where $f$ is the multiplicative function defined in \eqref{eq:f_right}, which depends on the specific value of $\sigma$, and on the parameter $L = \sqrt{\log X}$. 

For every real number $\sigma \in (\frac{1}{2},1)$ let us consider the following two moments:
\begin{align}
    M_1(R,X,\sigma) &= \sum_{\frac{X}{8}\leq d\leq \frac{5X}{16}}\mu(2d)^2R(8d)^2 L(\sigma,\chi_{8d}) \label{eq:FirstMoment_right} \\
    M_2(R,X,\sigma) &= \sum_{\frac{X}{8}\leq d\leq \frac{5X}{16}} \mu(2d)^2 R(8d)^2 (L(\sigma,\chi_{8d}))^2. \label{eq:SecondMoment_right} 
\end{align}
Analogously to \cref{sec:cricial_center}, we want to compute the asymptotic behavior of those moments and to prove \cref{Bogomolov2} using \eqref{eq:Titu}, as we did in the proof of \cref{Bogomolov}. 
To determine the aforementioned asymptotic behavior of \eqref{eq:FirstMoment_right} and \eqref{eq:SecondMoment_right}, we will need the following version of Rankin's trick, which resembles the estimates \eqref{eq:tail_main} and \eqref{eq:tail_error} appearing in the previous section.

\begin{lemma}[Rankin's Trick]\label{rankin's trick}
Fix $\sigma \in (\frac{1}{2},1)$, and let $D>0$. Then
\begin{align*}
\sum_{\substack{a,r,s\\ (ar>X) \vee (as>X)}}
\mu(a)^2f(a)^2\frac{\mu(r)^2f(r)\epsilon_D(r)}{r^\sigma}\frac{\mu(s)^2f(s)\epsilon_D(s)}{s^\sigma}=o\left(\prod_{p\text{ prime}}\left(1+f(p)^2-\frac{2Df(p)}{p^\sigma}\right)\right),
\end{align*}
where $\mu$ is the Möbius function, defined in \eqref{eq:mobius}, while $f$ is defined in \eqref{eq:f_right} and $\epsilon_D$ is defined in \eqref{eq:epsilon}.
\end{lemma}
\begin{proof}
We clearly have that
\begin{align*}
    &\sum_{\substack{a,r,s\\ (ar>X) \vee (as>X)}}
\mu(a)^2f(a)^2\frac{\mu(r)^2f(r)\epsilon_D(r)}{r^\sigma}\frac{\mu(s)^2f(s)\epsilon_D(s)}{s^\sigma} \\ \leq  &\left( \sum_{\substack{a,r,s\\ ar>X}}
\mu(a)^2f(a)^2\frac{\mu(r)^2f(r)\epsilon_D(r)}{r^\sigma}\frac{\mu(s)^2f(s)\epsilon_D(s)}{s^\sigma} \right) + \left( \sum_{\substack{a,r,s\\ as>X}}
\mu(a)^2f(a)^2\frac{\mu(r)^2f(r)\epsilon_D(r)}{r^\sigma}\frac{\mu(s)^2f(s)\epsilon_D(s)}{s^\sigma} \right),
\end{align*}
and, by symmetry, we may consider only the first sum appearing on the right-hand side. Now, notice that
\begin{align*}
    \sum_{\substack{a,r,s\\ ar>X}}
\mu(a)^2f(a)^2\frac{\mu(r)^2f(r)\epsilon_D(r)}{r^\sigma}\frac{\mu(s)^2f(s)\epsilon_D(s)}{s^\sigma} &= \left( \sum_{s} \frac{\mu(s)^2f(s)\epsilon_D(s)}{s^\sigma} \right) \cdot \left( \sum_{\substack{a,r\\ ar>X}} \mu(a)^2f(a)^2\frac{\mu(r)^2f(r)\epsilon_D(r)}{r^\sigma} \right) \\ &= \left( \prod_{p \ \text{prime}} \left( 1 + \frac{D f(p)}{p^\sigma} \right) \right) \cdot \left( \sum_{\substack{a,r\\ ar>X}} \mu(a)^2f(a)^2\frac{\mu(r)^2f(r)\epsilon_D(r)}{r^\sigma} \right)
\end{align*}
as follows easily from the definition of $\epsilon_D$ given in \eqref{eq:epsilon}. Moreover, if $\alpha \in (0,2\sigma-1)$ we have that
\begin{align*}
    \sum_{\substack{a,r\\ ar>X}} \mu(a)^2f(a)^2\frac{\mu(r)^2f(r)\epsilon_D(r)}{r^\sigma} &= \sum_{\substack{a,r\\ ar>X}} \mu(a)^2f(a)^2 r^{-\alpha} \frac{\mu(r)^2f(r)\epsilon_D(r)}{r^{\sigma - \alpha}} \\
    &< X^{-\alpha} \sum_{\substack{a,r\\ ar>X}} \mu(a)^2f(a)^2 a^{\alpha} \frac{\mu(r)^2f(r)\epsilon_D(r)}{r^{\sigma - \alpha}} \\
    &\leq X^{-\alpha} \sum_{a,r} \mu(a)^2f(a)^2 a^{\alpha} \frac{\mu(r)^2f(r)\epsilon_D(r)}{r^{\sigma - \alpha}} \\
    &= X^{-\alpha} \prod_{p \ \text{prime}} \left( 1 + f(p)^2 p^\alpha + \frac{D f(p)}{p^{\sigma - \alpha}} \right)
\end{align*}
because $ar > X$ and therefore $r^{-\alpha} < X^{-\alpha} a^\alpha$. 
This implies that
\begin{align*}
    &\phantom{\leq} \left( \sum_{\substack{a,r,s\\ ar>X}}
\mu(a)^2f(a)^2\frac{\mu(r)^2f(r)\epsilon_D(r)}{r^\sigma}\frac{\mu(s)^2f(s)\epsilon_D(s)}{s^\sigma} \right) \, \left( \prod_{p \ \text{prime}}\left(1+f(p)^2-\frac{2Df(p)}{p^ \sigma}\right)\right)^{-1} \\ 
     &\leq \ X^{-\alpha} \prod_{p \ \text{prime}} \left( 1+\frac{D f(p)}{p^\sigma} \right) \, \left( 1 + f(p)^2 p^\alpha + \frac{D f(p)}{p^{\sigma - \alpha}} \right) \left(1+f(p)^2-\frac{2Df(p)}{p^ \sigma}\right)^{-1} \\
     &\ll \exp\left(-\alpha\log X+\sum_{p \ \text{prime}}\left(f(p)^2(p^\alpha-1)+4Df(p)p^{\alpha-\sigma}\right)\right)\\
&\ll \exp\left(-\alpha \log X \right),
\end{align*}
where the last inequality can be obtained using the prime number theorem, as was done to prove the last equality appearing in \eqref{eq:main_theorem_from_moments}.
\end{proof}

\textbf{First moment:}
In order to bound the first moment $M_1(R,X,\sigma)$ from below, we will approximate $L(\sigma,\chi_{8d})$ with the infinite sum \begin{equation} \label{eq:twisted_L_right}
    \sum_{n=1}^{\infty} \frac{\chi_{8d}(n)}{n^\sigma} e^{-\frac{n}{X^2}},
\end{equation}
which can be regarded as a ``twisted version'' of the Dirichlet series defining $L(s,\chi_{8d})$, evaluated at $s = \sigma$. More precisely, the Dirichlet series $\sum_{n=1}^{+\infty} \chi_{8d}(n) n^{-s}$ clearly does not converge for $s = \sigma$, because $\sigma \in \left( \frac{1}{2},1 \right)$. However, the exponential factor $\exp(-nX^{-2})$ appearing in \eqref{eq:twisted_L_right} makes this series converge. Furthermore, the value of \eqref{eq:twisted_L_right} is sufficiently close to $L(s,\chi_{8d})$, as shown by the following result.

\begin{proposition}
\label{L-function approximation}
For every $\sigma \in (\frac{1}{2},1)$ and $\delta \in \left(0, \sigma-\frac{1}{2}\right)$ there exists two positive constants $c_8(\sigma,\delta)$ and $c_9(\delta)$ such that for every $X \geq c_9(\delta)$ we have that
\begin{equation}\label{error}
\sum_{\frac{X}{8}\leq d\leq \frac{5X}{16}}\mu(2d)^2R(8d)^2\left|L(\sigma,\chi_{8d})-\sum_{n=1}^{\infty}\frac{\chi_{8d}(n)}{n^\sigma} e^{-\frac{n}{X^2}}\right|^2 \leq c_8(\sigma,\delta) X^{1-4\delta}\prod_{p\text{ prime}}\left(1+f(p)^2+4\frac{f(p)}{p^{\sigma-\delta}}\right).
\end{equation}
In particular, we have that
\[
\left\lvert M_1(R,X,\sigma) - \sum_{\frac{X}{16}\leq d\leq \frac{X}{8}}\mu(2d)^2R(8d)^2\sum_{n=1}^{\infty}\frac{\chi_{8d}(n)}{n^\sigma} e^{-\frac{n}{X^2}} \right\rvert \leq c_8(\sigma,\delta) X^{1-4\delta}\prod_{p\text{ prime}}\left(1+f(p)^2+4\frac{f(p)}{p^{\sigma-\delta}}\right).
\].
\end{proposition}
\begin{proof}
Fix some $\delta \in \left( 0, \sigma - \frac{1}{2} \right)$. We use the following equality 
\begin{equation}\label{Exponential}
e^{-y}=\frac{1}{2\pi i}\int_{(b)}\Gamma(w)y^{-w}dw,
\end{equation}
which is valid for every $y,b>0$. Taking $b>1$ and  expanding $L(w,\chi_{8d})$ into its Dirichlet series for $\text{Re}(w)=b$ we have that
\begin{equation*}
\sum_{n=1}^{\infty}\frac{\chi_{8d}(n)}{n^{\sigma}} e^{-\frac{n}{X^2}}=\frac{1}{2\pi i}\int_{(b)}\frac{\Gamma(w-\sigma)}{X^{2(\sigma-w)}}L(w,\chi_{8d})dw.
\end{equation*}
Moving the integration contour from $(b)$ to $(c)$, where $c := \sigma - \delta$, we have a simple pole at $w=\sigma$ with residue $L(\sigma,\chi_{8d})$. Hence
\begin{align*}
L(\sigma,\chi_{8d})-\sum_{n=1}^{\infty}\frac{\chi_{8d}(n)}{n^{\sigma}} e^{-\frac{n}{X^2}}=\frac{1}{2\pi i}\int_{(c)}\frac{\Gamma(w-\sigma)}{X^{2(\sigma-w)}}L(w,\chi_{8d})dw.
\end{align*}

Applying the Cauchy-Schwarz integral inequality,
we notice, therefore that
\begin{align}
    \notag &\phantom{\leq} \sum_{\frac{X}{8}\leq d\leq \frac{5X}{16}}\mu(2d)^2R(8d)^2\left|L(\sigma,\chi_{8d})-\sum_{n=1}^{\infty}\frac{\chi_{8d}(n)}{n^\sigma} e^{-\frac{n}{X^2}}\right|^2 \\ 
    \notag &= \frac{1}{4\pi^2} \sum_{\frac{X}{8}\leq d\leq \frac{5X}{16}}\mu(2d)^2R(8d)^2\left| \int_{(c)}\frac{\Gamma(w-\sigma)}{X^{2(\sigma-w)}}L(w,\chi_{8d})dw \right|^2 
    \\
    &\leq \frac{1}{4\pi^2} \sum_{\frac{X}{8}\leq d\leq \frac{5X}{16}}\mu(2d)^2R(8d)^2 \left( \int_{(c)} \lvert \Gamma(w-\sigma) \rvert dw \right) \left( \int_{(c)} \frac{\lvert \Gamma(w-\sigma) \rvert}{\lvert X^{2 (\sigma - w)} \rvert^2} \lvert L(w,\chi_{8d}) \rvert^2 dw \right) 
    \\
    \notag &= \frac{1}{4\pi^2 X^{4\delta}} \left( \int_{-\infty}^{+\infty} \lvert \Gamma(it-\delta) \rvert dt \right) \sum_{n_1,n_2 \leq X} r(n_1) r(n_2) \left( \int_{(c)} \lvert \Gamma(w-\sigma) \rvert \left( \sum_{\frac{X}{8} \leq d \leq \frac{5 X}{16}} \mu(2d)^2 \chi_{8d}(n_1 n_2) \lvert L(w,\chi_{8d}) \rvert^2 \right) dw \right). 
\end{align}
Observe that the sum appearing in the last integral runs over those integers $d$ such that $1 \leq \frac{8d}{X} \leq \frac{5}{2}$. Therefore, if $\Phi$ denotes the function defined in \eqref{eq:Phi} we have that
\[
    \sum_{\frac{X}{8} \leq d \leq \frac{5 X}{16}} \mu(2d)^2 \chi_{8d}(n_1 n_2) \lvert L(w,\chi_{8d}) \rvert^2 \leq e^{\frac{1}{2}} \sum_{\frac{X}{8} \leq d \leq \frac{5 X}{16}} \mu(2d)^2 \chi_{8d}(n_1 n_2) \lvert L(w,\chi_{8d}) \rvert^2 \Phi\left( \frac{8d}{X} \right) = e^{\frac{1}{2}} M_{w-\frac{1}{2},\overline{w}-\frac{1}{2}}\left( n_1 n_2; \frac{X}{8} \right),
\]
where $M$ is the function defined in \eqref{eq:M_Sono}.
Hence, if $c_{10}(\delta)$ is the constant defined in \eqref{eq:c10}, which depends only on $\delta$, we have that
\begin{equation} \label{eq:bound_right_L-regularized}
    \begin{aligned}
        &\phantom{\leq} \sum_{\frac{X}{8}\leq d\leq \frac{5X}{16}}\mu(2d)^2R(8d)^2\left|L(\sigma,\chi_{8d})-\sum_{n=1}^{\infty}\frac{\chi_{8d}(n)}{n^\sigma} e^{-\frac{n}{X^2}}\right|^2 \\ 
        &\leq c_{10}(\delta) X^{-4\delta} \sum_{n_1,n_2 \leq X} r(n_1) r(n_2) \left( \int_{(c)} \lvert \Gamma(w-\sigma) \rvert M_{w-\frac{1}{2},\overline{w}-\frac{1}{2}}\left( n_1 n_2;\frac{X}{8} \right) dw \right) \\
        &= c_{10}(\delta) X^{-4\delta} \sum_{\substack{a,r,s\\ar,as\leq X\\(a,r)=(a,s)=(r,s)=1}}
\mu(a)^2f(a)^2\mu(r)^2f(r)\mu(s)^2f(s) \left( \int_{(c)} \lvert \Gamma(w-\sigma) \rvert M_{w-\frac{1}{2},\overline{w}-\frac{1}{2}}\left( a^2 r s;\frac{X}{8} \right) dw \right),
    \end{aligned}
\end{equation}
where in the last equality, we performed the same substitution appearing in \eqref{equ3}.

To proceed, we need to recall a result of Sono \cite[Equation~(4.80)]{sono2020second}, which implies that there exist two positive absolute constants $c_{11}$ and $c_{12}$ such that for every $w \in \mathbb{C}$ with $\mathrm{Re}(w) = c$, every $X$ and every $a,r,s$ we have that
\begin{equation} \label{eq:bound_M_Sono}
    M_{w-\frac{1}{2},\overline{w}-\frac{1}{2}}\left( a^2 r s;\frac{X}{8} \right) \leq \frac{c_{11}}{\sqrt{rs}} \left( \int_0^{+\infty} \Phi\left(\frac{8t}{X}\right) dt \right) A_{w-\frac{1}{2},\overline{w}-\frac{1}{2}}(a^2rs) \leq c_{12} X \frac{1} {\sqrt{rs}} A_{c-\frac{1}{2},c-\frac{1}{2}}(a^2 r s),
\end{equation}
where $A$ is the function defined in \eqref{eq:A}.
Using this definition, we see that
\begin{equation}
\label{eq:bound_A_Sono}
    \frac{1}{\sqrt{rs}} A_{c-\frac{1}{2},c-\frac{1}{2}}(a^2 r s) = \frac{1}{\sqrt{rs}} \sum_{n \ \text{odd}} \frac{1}{n} \frac{\tau(rsn^2)}{(rsn^2)^{c-\frac{1}{2}}} \prod_{\substack{p \ \text{prime} \\p \mid arsn}} \left( \frac{p}{p+1} \right) \leq \frac{\tau(rs)}{(rs)^c} \sum_{n \ \text{odd}} \frac{\tau(n^2)}{n^{2c}} = \frac{\tau(rs)}{(rs)^c} \frac{\zeta^3(2c)}{\zeta(4c)}.
\end{equation}
Therefore, combining \eqref{eq:bound_right_L-regularized} with \eqref{eq:bound_M_Sono} and \eqref{eq:bound_A_Sono} we see that there exists a constant $c_8(\sigma,\delta) > 0$, depending on $\sigma$ and $\delta$, such that
\[
\begin{aligned}
    &\phantom{c_8(\delta)} \sum_{\frac{X}{8}\leq d\leq \frac{5X}{16}}\mu(2d)^2R(8d)^2\left|L(\sigma,\chi_{8d})-\sum_{n=1}^{\infty}\frac{\chi_{8d}(n)}{n^\sigma} e^{-\frac{n}{X^2}}\right|^2 \\ 
    &\leq c_8(\sigma,\delta) X^{1-4\delta} \sum_{\substack{a,r,s\\ar,as\leq X\\(a,r)=(a,s)=(r,s)=1}}
\mu(a)^2f(a)^2\frac{\mu(r)^2f(r)\tau(r)}{r^{\sigma - \delta}}\frac{\mu(s)^2 f(s) \tau(s)}{s^{\sigma - \delta}}.
\end{aligned}
\]
To conclude, we extend the sum over $a$, $r$, and $s$ to run over all integers, since \cref{rankin's trick} with $D=2$ shows that the tail of this sum is little-$o$ of the right-hand side of \eqref{error}. Once we have done this, we can observe that
\[
\begin{aligned}
    \sum_{\substack{a,r,s\\(a,r)=(a,s)=(r,s)=1}}\mu(a)^2f(a)^2\frac{\tau(r)\mu(r)^2f(r)}{r^{\sigma-\delta}}\frac{\mu(s)^2f(s)\tau(s)}{s^{\sigma-\delta}} &= \prod_{p\text{ prime}}\left(1+f(p)^2+ 2\frac{f(p)\tau(p)}{p^{\sigma-\delta}}\right) \\
    &= \prod_{p\text{ prime}}\left(1+f(p)^2+ 4\frac{f(p)}{p^{\sigma-\delta}}\right),
\end{aligned}
\]
which allows us to conclude.
\end{proof}

The following lemma gives the asymptotic behavior of the expression that one obtains from $M_1(R,X,\sigma)$ by replacing $L(\sigma,\chi_{8d})$ with the sum $\sum_{n=1}^{\infty}\chi_{8d}(n)n^{-\sigma}\exp(-n X^{-2})$.

\begin{lemma}\label{M1 approximation}
Let $\sigma \in \left( \frac{1}{2},1 \right)$. Then, we have that
\begin{equation}\label{eq M1 approx}
\sum_{\frac{X}{8}\leq d\leq \frac{5X}{16}}\mu(2d)^2R(8d)^2\sum_{n=1}^{\infty}\frac{\chi_{8d}(n)}{n^\sigma} e^{-\frac{n}{X^2}} \sim c_{13}(\sigma) X\prod_{p\text{ prime}}\left(1+f(p)^2-2\frac{f(p)}{p^{\sigma}}\right)   
\end{equation}
as $X \to +\infty$, where $c_{13}(\sigma)$ is some positive constant depending only on $\sigma$.
\end{lemma}
\begin{proof}
By substituting the definition of $R(8d)$ we have
\[
\sum_{\frac{X}{8}\leq d\leq \frac{5X}{16}}\mu(2d)^2R(8d)^2\sum_{n=1}^{\infty}\frac{\chi_{8d}(n)}{n^\sigma} e^{-\frac{n}{X^2}} = \sum_{n_1,n_2\leq X}r(n_1)r(n_2)\sum_{n=1}^{\infty}\frac{e^{-\frac{n}{X^2}}}{n^{\sigma}}\sum_{\frac{X}{8}\leq d\leq \frac{5X}{16}}\mu(2d)^2\chi_{8d}(nn_1n_2).
\]
Applying the Polya--Vinogradov inequality, as done in the proof of \cref{SecondMoment} and of \cite[Lemma~3.2]{Soundararajan_2008}, we see that the right-hand side of the previous equation is asymptotically equivalent to
\begin{equation}\label{sum chi}
\frac{3X}{8\zeta(2)}\sum_{\substack{a,r,s\\ar,as\leq X\\(a,r)=(a,s)=(r,s)=1}}
\mu(a)^2f(a)^2\frac{\mu(r)f(r)\mu(s)f(s)}{(rs)^{\sigma}}\sum_{m \text{ odd}}\frac{e^{-\frac{rsm^2}{X^2}}}{m^{2\sigma}}\prod_{\substack{p \ \text{prime} \\ p\mid arsm}}\left(\frac{p}{p+1}\right),
\end{equation}
as $X \to +\infty$.
Let us first focus on the interior sum over $m$. Thanks to \eqref{Exponential}, we see that
\begin{equation}\label{integral}
\begin{aligned}
&\sum_{m \text{ odd}}\frac{e^{-\frac{rsm^2}{X^2}}}{m^{2\sigma}}\prod_{\substack{p \ \text{prime} \\ p\mid 2arsm}} \left(\frac{p}{p+1}\right) \\ = &\frac{1}{2\pi i} \int_{(b)}\Gamma(w)\left(\frac{X^2}{rs}\right)^{w}\sum_{m \text{ odd}}\frac{1}{m^{2(\sigma+w)}}\prod_{\substack{p \ \text{prime} \\ p\mid arsm}} \left(\frac{p}{p+1}\right)dw \\
= &\frac{1}{2\pi i} \left( \prod_{\substack{p \ \text{prime} \\ p\mid ars}} \left(\frac{p}{p+1}\right) \right) \int_{(b)}\Gamma(w)\left(\frac{X^2}{rs}\right)^{w} \zeta(2(\sigma+w))\left(1-\frac{1}{2^{2(\sigma+w)}}\right) \prod_{\substack{p \ \text{prime} \\ p\nmid 2ars}} \left(1-\frac{1}{p^{2(\sigma+w)}(p+1)}\right) dw.
\end{aligned}
\end{equation}

Moving the line of integration from $(b)$ to $(-\gamma)$ for some $\gamma < \sigma - \frac{1}{2}$, similarly to what we did in \eqref{eq:residue} following \cite[Page 465]{Soundararajan_2000}, we see that
\begin{align*}
    &\int_{(b)}\Gamma(w)\left(\frac{X^2}{rs}\right)^{w} \zeta(2(\sigma+w))\left(1-\frac{1}{2^{2(\sigma+w)}}\right) \prod_{\substack{p \ \text{prime} \\ p\nmid 2ars}} \left(1-\frac{1}{p^{2(\sigma+w)}(p+1)}\right) dw \\
    = &\Res_{w=0} \left( \Gamma(w)\left(\frac{X^2}{rs}\right)^{w}\zeta(2(\sigma+w))\left(1-\frac{1}{2^{2(\sigma+w)}}\right)\prod_{\substack{p \ \text{prime} \\ p\nmid 2ars}} \left(1-\frac{1}{p^{2(\sigma+w)}(p+1)}\right) \right) + O\left( \left( \frac{rs}{X^2} \right)^\gamma \right) \\
    = &\zeta(2\sigma)\left(1-\frac{1}{2^{2\sigma}}\right) \prod_{\substack{p \ \text{prime} \\ p\nmid 2ars}} \left(1-\frac{1}{p^{2\sigma}(p+1)}\right) + O\left( \left( \frac{rs}{X^2} \right)^\gamma \right). 
\end{align*}
Since $rs \leq (X/a)^2$, this implies that \eqref{sum chi} is asymptotically equivalent to 
\[
\frac{3X}{8\zeta(2)}\zeta(2\sigma)\left(1-\frac{1}{2^{2\sigma}}\right)\sum_{\substack{a,r,s\\ar,as\leq X\\(a,r)=(a,s)=(r,s)=1}}
\mu(a)^2f(a)^2\frac{\mu(r)f(r)\mu(s)f(s)}{(rs)^{\sigma}}
\prod_{\substack{p \ \text{prime} \\ p\mid ars}} \left(\frac{p}{p+1}\right)\prod_{\substack{p \ \text{prime} \\ p\nmid 2ars}} \left(1-\frac{1}{p^{2\sigma}(p+1)}\right).
\]
We can now extend the last sum over $a,r,s$ to run over all the integers, since \cref{rankin's trick} with $D=1$ shows that the tail of this sum is little-$o$ of the right-hand side of \eqref{eq M1 approx}.
Then, we see by multiplicativity that there exists a constant $c_9(\sigma)$, depending only on $\sigma$, such that
\begin{align*}
    &\frac{3X}{8\zeta(2)}\zeta(2\sigma)\left(1-\frac{1}{2^{2\sigma}}\right)\sum_{\substack{a,r,s\\(a,r)=(a,s)=(r,s)=1}}
\mu(a)^2f(a)^2\frac{\mu(r)f(r)\mu(s)f(s)}{(rs)^{\sigma}}\prod_{\substack{p \ \text{prime} \\ p\mid ars}}\left(\frac{p}{p+1}\right)\prod_{\substack{p \ \text{prime} \\ p\nmid 2ars}}\left(1-\frac{1}{p^{2\sigma}(p+1)}\right) \\
= &c_9(\sigma) X \prod_p \left(1+f(p)^2- 2\frac{f(p)}{p^{\sigma}}\right), 
\end{align*}
and this yields the desired result.
\end{proof}

The following theorem gives a lower bound for $M_1(R,X,\sigma)$.

\begin{theorem}\label{M1alpha}
Let $\sigma \in \left( \frac{1}{2},1 \right)$. Then, there exist two positive constants $c_{14}(\sigma)$ and $c_{15}(\sigma)$ such that
\begin{equation*}
M_1(R,X,\sigma)
\geq c_{14}(\sigma) X\prod_{p\text{ prime}}\left(1+f(p)^2-2\frac{f(p)}{p^{\sigma}}\right),
\end{equation*}
for every $X \geq c_{15}(\sigma)$, where $M_1(R,X,\sigma)$ is defined as in \eqref{eq:FirstMoment_right}.
\end{theorem}
\begin{proof}
By \cref{L-function approximation} and \cref{M1 approximation} it is enough to prove that
\begin{equation} \label{eq:right_firstmoment_final}
    \frac{1}{X^{4\delta}}\prod_{p\text{ prime}}\frac{\left(1+f(p)^2+4\frac{f(p)}{p^{\sigma-\delta}}\right)}{\left(1+f(p)^2-2\frac{f(p)}{p^{\sigma}}\right)}=o(1),
\end{equation}
for some positive real number $\delta$ satisfying $\delta<\sigma-\frac{1}{2}$.
Notice that 
\begin{equation}\label{asymM1}
\begin{aligned}
\prod_{p \ \text{prime}} \frac{\left(1+f(p)^2+4\frac{f(p)}{p^{\sigma-\delta}}\right)}{\left(1+f(p)^2-2\frac{f(p)}{p^{\sigma}}\right)}
&=\prod_{p \ \text{prime}}\left(1+\frac{4\frac{f(p)}{p^{\sigma-\delta}}+2\frac{f(p)}{p^\sigma}}{1+f(p)^2-2\frac{f(p)}{p^{\sigma}}}\right) \\
&\ll \exp\left((6+o(1))L\sum_{\substack{p \ \text{prime} \\ p\geq L^{\frac{1}{\sigma}}}}\frac{1}{p^{2\sigma-\delta}}\right) \\ &= \exp\left((6+o(1))\frac{L}{ L^{\frac{2\sigma-\delta-1}{\sigma}}\log( L^{\frac{1}{\sigma}})}\right)\\ &=\exp\left(\left(12\sigma+o(1)\right)\frac{\log^{\frac{1+\delta-\sigma}{2\sigma}} X}{\log\log X}\right).
\end{aligned}
\end{equation}
Notice that $1+\delta-\sigma<1$ and therefore $\exp\left(\left(12\sigma+o(1)\right)\frac{\log^{\frac{1+\delta-\sigma}{2\sigma}} X}{\log\log X}\right) = o(X^{\varepsilon})$ for every $\varepsilon>0$, and in particular for $\varepsilon=4\delta$, this computation shows that \eqref{eq:right_firstmoment_final} holds true, and therefore concludes the proof.
\end{proof}
\begin{remark}
    Note the infinite sum $\sum_{p \ \text{prime}, p \geq L^{1/\sigma}} p^{-2 \sigma}$ appearing in \eqref{asymM1} diverges when $\sigma = \frac{1}{2}$.
    This shows, in particular, that the proofs appearing in this section cannot be ``continuously deformed'' to reach the value $\sigma = \frac{1}{2}$. This can also be noticed in the choice of the resonator coefficients. Indeed, in this section, the multiplicative function $f$ vanishes only on finitely many primes and is non-zero for each prime $p \geq L^{\frac{1}{\sigma}}$. On the other hand, we saw in the previous section that when $\sigma = \frac{1}{2}$, we need to consider a multiplicative function $f$ which vanishes at \textit{all but finitely many} primes. More precisely, $f(p) = 0$ unless $L^{\frac{1}{\sigma}} = L^2 \leq p \leq \exp(\log^2(L)) = \exp(\log^{\frac{1}{\sigma}}(L))$, where $L = \sqrt{\log N \log\log N}$. Therefore, the support of the function $f$ and the shape of the parameter $L$ in the case $\sigma = \frac{1}{2}$ are pretty different from the ones considered in the present section.
\end{remark}
\textbf{Second moment:} Using a result of Sono \cite{sono2020second}, we find the asymptotic behavior for the second moment $M_2(R,X,\sigma)$.

\begin{proposition}\label{M2alpha}
Let $\sigma \in \left( \frac{1}{2}, 1 \right)$. Then, there exist two positive constants $c_{16}(\sigma)$ and $c_{17}(\sigma)$ such that
\begin{equation}\label{M2 assymptotic}
M_2(R,X,\sigma) \leq c_{16}(\sigma) X\prod_{p \text{ prime}} \left(1+f(p)^2-4\frac{f(p)}{p^{\sigma}}\right),
\end{equation}  
for every $X \geq c_{17}(\sigma)$.
\end{proposition}
\begin{proof}
By substituting the definitions of $R(8d)$ and $r(n)$ into $M_2(R,X,\sigma)$ we have that 
\[
M_2(R,X,\sigma) = \sum_{\substack{a,r,s\\ar,as\leq X\\(a,r)=(a,s)=(r,s)=1}}
\mu(a)^2f(a)^2\mu(r)f(r)\mu(s)f(s)
\sum_{\frac{X}{8}\leq d\leq \frac{5X}{16}}\mu(2d)^2L(\sigma,\chi_{8d})^2\chi_{8d}(a^2rs).
\]  
To compute the inner sum, we notice that 
\[
    \sum_{\frac{X}{8}\leq d\leq \frac{5X}{16}}\mu(2d)^2L(\sigma,\chi_{8d})^2\chi_{8d}(a^2rs) = M_{\sigma-\frac{1}{2},\sigma-\frac{1}{2}}(a^2rs),
\]
where $M_{\alpha_1,\alpha_2}(l)$ is defined as in \eqref{eq:M_Sono}.
This function was studied by Sono, who proved in \cite[Theorem~2.2]{sono2020second} (see also the proof of \cite[Theorem~5.8]{GeLa24}) that
\[
    M_{\sigma - \frac{1}{2},\sigma - \frac{1}{2}}(a^2 r s) \sim c_{18}(\sigma) \frac{X}{(rs)^\sigma} \sum_{m \ \text{odd}} \frac{\tau(r s m^2)}{m^{2 \sigma}} \prod_{\substack{p \ \text{prime}\\p \mid m a r s}} \left( \frac{p}{p+1} \right)
\]
as $X \to +\infty$, where $c_{18}(\sigma)$ is a constant depending only on $\sigma$.
Therefore, we see that
\begin{align*}
    M_2(R,X,\sigma) &\sim c_{18}(\sigma) X \sum_{\substack{a,r,s\\ar,as\leq X\\(a,r)=(a,s)=(r,s)=1}}
\mu(a)^2f(a)^2\frac{\mu(r)f(r)}{r^\sigma}\frac{\mu(s)f(s)}{s^\sigma}
\sum_{m \text{ odd}}\frac{\tau(rsm^2)}{m^{2\sigma}}\prod_{\substack{p \ \text{prime}\\p\mid mars}}\left(\frac{p}{p+1}\right) \\
&= c_{19}(\sigma) X \sum_{\substack{a,r,s\\ar,as\leq X\\(a,r)=(a,s)=(r,s)=1}}
\mu(a)^2f(a)^2\frac{\tau(r)\mu(r)f(r)}{r^\sigma}\frac{\tau(s)\mu(s)f(s)}{s^\sigma}
\eta(2\sigma; a^2rs) , 
\end{align*}
where the second equality follows by an application of \eqref{Sound Lemma 5.1} with $w = \sigma - \frac{1}{2}$, and therefore $c_{19}(\sigma) = c_{18}(\sigma) \cdot \zeta(2 \sigma)^3$.
To conclude, note that the asymptotic behavior of the last sum does not change if we extend it over all the integers $a,r,s$, thanks to \cref{rankin's trick} applied with $D=2$, which shows that the tail of this sum is little-$o$ of the right-hand side of \eqref{M2 assymptotic}. 
Moreover, recall that $\eta(\alpha;l) \leq 1$ for every integer $l$ and every $\alpha > 1$, as follows from the definitions \eqref{eq:eta_p}.
Therefore, there exist two positive constants $c_{16}(\sigma)$ and $c_{17}(\sigma)$ such that for every $X \geq c_{17}(\sigma)$ we have that
\begin{align*}
    M_2(R,X,\sigma) &\leq c_{16}(\sigma) X \sum_{\substack{a,r,s \\(a,r)=(a,s)=(r,s)=1}}
\mu(a)^2f(a)^2\frac{\tau(r)\mu(r)f(r)}{r^\sigma}\frac{\tau(s)\mu(s)f(s)}{s^\sigma} \\
&= c_{16}(\sigma) X\prod_{p\text{ prime}}\left(1+f(p)^2- 2\frac{f(p)\tau(p)}{p^{\sigma}}\right) = c_{16}(\sigma) X\prod_{p\text{ prime}}\left(1+f(p)^2- 4\frac{f(p)}{p^{\sigma}}\right),
\end{align*}
which allows us to conclude.
\end{proof}

\textbf{Final step:} The following corollary shows that for every $\sigma \in \left(\frac{1}{2},1\right)$ the function $h_{\sigma}$ does not have the Bogomolov property.

\begin{proof}[Proof of \cref{Bogomolov2}]
Using \eqref{eq:Titu} we see that
\begin{equation*}
\frac{\displaystyle{\sum_{\frac{X}{16}\leq d\leq \frac{X}{8}}\mu(2d)^2R(8d)^2 L(\sigma,\chi_d) ^2}}{\displaystyle{\sum_{\substack{\frac{X}{16}\leq d\leq \frac{X}{8}\\ L(\sigma,\chi_d)\neq 0 }}\mu(2d)^2R(8d)^2}}\leq \left(\frac{M_2(R,X,\sigma)}{M_1(R,X,\sigma)}\right)^2.  
\end{equation*}
By \cref{M1alpha} and \cref{M2alpha} we have that 
\begin{align*}
\left(\frac{M_2(R,X,\sigma)}{M_1(R,X,\sigma)}\right)^2 &\le
\frac{c_{16}(\sigma)}{c_{14}(\sigma)}\left(\prod_{p\text{ prime}}\frac{1+f(p)^2-4\frac{f(p)}{p^{\sigma}}}{1+f(p)^2-2\frac{f(p)}{p^{\sigma}}}\right)^2
\\
&=\frac{c_{16}(\sigma)}{c_{14}(\sigma)} \exp\left(-
\left(
8\sigma+o(1)\right)\frac{\log^{\frac{1-\sigma}{2\sigma}} X}{\log\log X}\right),
\end{align*}
which gives the desired result.
\end{proof}

\subsection{Random Euler products}    

    In the previous subsection we have proven that for every $\sigma \in \mathbb{R}$ such that $\frac{1}{2} < \sigma < 1$, the restriction of the Dedekind height $h_\sigma \colon \mathcal{N} \to \mathbb{R}$ to the subset $\mathcal{Q} \subseteq \mathcal{N}$ of isomorphism classes of quadratic fields does not satisfy the Bogomolov property. In this particular case, such a result suffices to conclude that $h_\sigma$ does not satisfy the Bogomolov property \textit{tout court}, because $h_\sigma([K]) := \lvert \zeta_K(\sigma) \rvert \geq 0$ for every $[K] \in \mathcal{N}$, and we have proven in the previous section that  $h_\sigma(\mathcal{Q})$ accumulates towards zero. This implies in particular that for every $B > 0$, the set $\{ K \in \mathcal{Q} \colon h_\sigma(K) < B \}$ is non-empty. Since this happens \textit{for every} $B>0$, these sets are, in fact, infinite. However, such a set could have density zero inside the set of all quadratic fields $\mathcal{Q}$, ordered by discriminant. The following proposition, proven by Lamzouri \cite{Lamzouri_2011}, shows that this is not the case and gives an expression for the aforementioned density, which is completely explicit up to a logarithmic error term.
    \begin{proposition}[Lamzouri] \label{prop:Lamzouri}
        Fix $\sigma \in \mathbb{R}$ such that $\frac{1}{2}  < \sigma < 1$. Then, we have that
        \[
            \begin{aligned}
                &\phantom{=} \lim_{X \to +\infty} \frac{\# \{ [K] \in \mathcal{Q} \colon \lvert \Delta_K \rvert \leq X, \ h_\sigma(K) \leq B \}}{\# \{ [K] \in \mathcal{Q} \colon \lvert \Delta_K \rvert \leq X\}} \\ &= \exp\left( - c_{20}(\sigma) \log\left( \frac{\zeta(\sigma)}{B} \right)^{\frac{1}{1-\sigma}} \log\left(\log\left( \frac{\zeta(\sigma)}{B} \right)\right)^{\frac{\sigma}{1-\sigma}} \left( 1 + O\left( \log\log\left(\frac{\zeta(\sigma)}{B}\right)^{-\frac{1}{2}} \right) \right) \right),
            \end{aligned}
        \]
        as $B \to 0$, where $c_{20}(\sigma)$ is the constant defined in \eqref{eq:c20}. 
    \end{proposition}
    \begin{proof}
        Lamzouri showed in \cite[Theorem~1.6]{Lamzouri_2011} that the distribution of the $L$-values $L(\sigma,\chi_d)$ as $d$ varies amongst all the fundamental discriminants resembles that of a random Euler product 
        \[ L(\sigma,\mathbf{X}) := \prod_{p \ \text{prime}} \left( 1- \frac{\mathbf{X}_p}{p^\sigma} \right)^{-1},\] where $\{\mathbf{X}_p\}_p$ denotes a family of independent random variables, each of which assumes the values $1$ and $-1$ with probability $\frac{1}{2}$. This resemblance is particularly true for the moments of these distributions, as proven in \cite[Proposition~5.2]{Lamzouri_2011}, and therefore for the tails of the absolute values of these distributions. Since the random variables $\mathbf{X}_p$ are symmetric, \textit{i.e.} $\mathbf{X}_p$ and $-\mathbf{X}_p$ are identically distributed, the upper and lower tails of the random variable $\lvert L(\sigma,\mathbf{X}) \rvert$ coincide, as observed in \cite[Remark~4]{Lamzouri_2011}, and their asymptotic value is determined by \cite[Theorem~1.9]{Lamzouri_2011}. In fact, this result, as well as \cite[Theorem~1.6]{Lamzouri_2011}, is more precise, providing as well an explicit error term for the distribution of the truncated Euler products \[\prod_{\substack{p \ \text{prime} \\ p \leq y}} (1-\mathbf{X}_p p^{-\sigma})^{-1},\] for any fixed $y>0$. Putting everything together, and letting $y \to +\infty$, we obtain a proof of \cref{prop:Lamzouri}.
    \end{proof}

    \begin{remark}
        As we already observed, we have that $\sup(h_\sigma(\mathcal{Q})) = +\infty$ for every $\sigma \in \mathbb{R}$ such that $\frac{1}{2} < \sigma < 1$. Moreover, \cite[Theorem~1.6]{Lamzouri_2011} shows that
        \[
            \begin{aligned}
                &\phantom{=} \lim_{X \to +\infty} \frac{\# \{ [K] \in \mathcal{Q} \colon \lvert \Delta_K \rvert \leq X, \ h_\sigma(K) \geq B \}}{\# \{ [K] \in \mathcal{Q} \colon \lvert \Delta_K \rvert \leq X\}} \\ &= \exp\left( - c_{20}(\sigma) \log\left( \frac{B}{\zeta(\sigma)} \right)^{\frac{1}{1-\sigma}} \log\left(\log\left( \frac{B}{\zeta(\sigma)} \right)\right)^{\frac{\sigma}{1-\sigma}} \left( 1 + O\left( \log\log\left(\frac{B}{\zeta(\sigma)}\right)^{-\frac{1}{2}} \right) \right) \right),
            \end{aligned}
        \]
        for every $B > 0$ which is sufficiently big with respect to $\sigma$.
    \end{remark}

    \begin{remark}
        We underline that \cref{prop:Lamzouri} does not assume the validity of the Generalized Riemann Hypothesis, even though other parts of Lamzouri's article \cite{Lamzouri_2011} do. 
    \end{remark}

    What happens when $\sigma = 1$? In this case, the sets $\{[K] \in \mathcal{Q} \colon h_1(K) \leq B\}$ have again positive density, as was shown by Granville and Soundararajan \cite{Granville_Soundararajan_2003}.

    \begin{proposition}[Granville \& Soundararajan] \label{prop:Granville_Soundararajan}
        When $B \to 0$ one has that
        \[
            \begin{aligned}
                &\phantom{=} \lim_{X \to +\infty} \frac{\# \{ [K] \in \mathcal{Q} \colon \lvert \Delta_K \rvert \leq X, \ h_1(K) \leq B \}}{\# \{ [K] \in \mathcal{Q} \colon \lvert \Delta_K \rvert \leq X\}} \\ &= \exp\left( - c_{21} \exp\left( \frac{\zeta(2)}{B e^{\gamma_0}} \right) \left( \frac{B e^{\gamma_0}}{\zeta(2)} \right) + O\left( \exp\left( \frac{\zeta(2)}{B e^{\gamma_0}} \right) \left( \frac{B e^{\gamma_0}}{\zeta(2)} \right)^2 \right) \right),
            \end{aligned}
        \]
        where $c_{21}$ is the absolute constant defined in \eqref{eq:21}, while $\gamma_0$ is the Euler-Mascheroni constant, defined by setting $j=0$ in \eqref{eq:Stieltjes}.
    \end{proposition}
\begin{proof}
    This result can be obtained by combining \cite[Proposition~1]{Granville_Soundararajan_2003}, which computes the distribution of some random Euler products evaluated at $s = 1$, together with \cite[Theorem~2]{Granville_Soundararajan_2003}, which implies, letting $x \to +\infty$, that the values of Dirichlet $L$-functions $L(1,\chi_d)$ have exactly the distribution of the random Euler products considered in \cite[Proposition~1]{Granville_Soundararajan_2003}.
\end{proof}

To conclude this section, let us note once more that combining \cref{prop:Lamzouri} and \cref{prop:Granville_Soundararajan} gives an alternative proof of \cref{Bogomolov2}, which moreover has the advantage of showing that the density of the set of quadratic fields $K$ such that $\lvert \zeta_K(\sigma) \rvert \leq B$ is always positive, for every $B > 0$ and every real number $\sigma$ such that $\frac{1}{2} < \sigma \leq 1$. 

\begin{remark}
    Note that
    \[
        \lim_{\sigma \to \frac{1}{2}} \int_0^{+\infty} \frac{\log(\cosh(x))}{x^{1+\frac{1}{\sigma}}} dx = \lim_{\sigma \to 1} \int_0^{+\infty} \frac{\log(\cosh(x))}{x^{1+\frac{1}{\sigma}}} dx = +\infty,
    \]
    which implies that $c_{20}(\sigma) \to 0$ as $\sigma \to \frac{1}{2}$, while $c_{20}(\sigma) \to +\infty$ as $\sigma \to 1$, where $c_{20}(\sigma)$ is the constant defined in \eqref{eq:c20}, which appears in \cref{prop:Lamzouri}. This does not imply that the sets $\{[K] \in \mathcal{Q} \colon h_{1/2}(K) \leq B \}$ have density zero,
    but it is just another indication of the fact that Lamzouri's proof cannot be ``continuously deformed'' to $\sigma = \frac{1}{2}$ or to $\sigma = 1$.
\end{remark}


    \section{Outside the critical strip}
    \label{sec:critical_right}
    
    This section aims to use explicit bounds for the Dedekind $\zeta$-functions $\zeta_K(s)$ in the region of absolute convergence, as well as their functional equations, to prove that for every $s \in \mathbb{C}$ the Dedekind height $h_s$ defined in \eqref{eq:Dedekind_height} does not have the Bogomolov (or Northcott) property if $s = \sigma \in \mathbb{R}$ and $\sigma > 1$. On the other hand, we will show that if we restrict this function to classes of number fields having bounded degree, then $h_s$ has even the Northcott property if $s \in \mathbb{C}$ is any complex number such that $\mathrm{Re}(s) \leq 0$. This generalizes some results of \cite{Pazuki_Pengo_2024}, which dealt with integers $s \in \mathbb{Z}$, and some results proven by Généreux and Lalín \cite{GeLa24}.

    \subsection{The right of the critical strip}\label{bigger 1}

    In this section, we study the Diophantine properties of the special values of Dedekind zeta functions at real numbers $\sigma > 1$. Since the Euler product defining $\zeta_K(\sigma)$ converges absolutely on this half line, we have $\zeta_K^\ast(\sigma) = \zeta_K(\sigma)$ for every $\sigma > 1$. In this case, \cref{s>1} claims that the Dedekind height $h_\sigma$ defined in \eqref{eq:Dedekind_height} does not satisfy the Bogomolov property, henceforth generalizing \cite[Theorem~3.2]{GeLa24}, which proved that it does not have the Northcott property. We devote the rest of this sub-section to giving proof of this result.
    
    \begin{proof}[Proof of \cref{s>1}]
        Note first of all that
        \[
            h_\sigma([K]) = \lvert \zeta_K^\ast(\sigma) \rvert = \lvert \zeta_K(\sigma) \rvert = \zeta_K(\sigma) = 1 + \sum_{n=2}^{+\infty} \frac{\# \{ I \subseteq \mathcal{O}_K \colon \lvert \mathcal{O}_K/I \rvert = n \}}{n^\sigma}
        \]
        for every $\sigma > 1$ and $[K] \in \mathcal{N}$. This shows immediately that $h_\sigma(\mathcal{N}) \subseteq \mathbb{R}_{> 1}$, as claimed.

        Let us now show that $\inf(h_\sigma(\mathcal{N})) = 1$, by constructing a sequence of number fields $\{ K_d \}_{d=1}^{+\infty}$ such that \[\zeta_{K_d}(\sigma) \approx \zeta(d \sigma)\] as $d \to +\infty$, which allows us to conclude because $\zeta(d \sigma) \to 1$ as $d \to +\infty$. 
        To construct the number fields $K_d$, let $p_1 = 2, p_2 = 3,\dots$ be the sequence of rational prime numbers. Then, for every pair of integers $d,n \geq 1$, there exists a monic polynomial $\psi_{d,n} \in \mathbb{Z}[x]$ which has degree $d$ and is irreducible modulo every prime $p_j$ for $j \in \{1,\dots,n\}$. To construct $\psi_{d,n}$, it suffices to choose some monic polynomials $\beta_1 \in \mathbb{F}_{p_1}[x],\dots,\beta_n \in \mathbb{F}_{p_n}[x]$ which have degree $d$ and are irreducible. Then, applying the Chinese remainder theorem coefficient-wise, we can construct a monic polynomial $\psi_{n,d} \in \mathbb{Z}[x]$ which has degree $d$ and is congruent to $\beta_j$ modulo each prime $p_j$. This polynomial $\psi_{d,n}$ is necessarily irreducible in $\mathbb{Z}[x]$ and defines a number field $F_{d,n} := \mathbb{Q}[x]/(\psi_{d,n})$ of degree $d$, such that $p_1,\dots,p_n$ are inert in the extension $\mathbb{Q} \subseteq F_{d,n}$. This implies that
        \[
            \zeta_{F_{d,n}}(\sigma) = \left( \prod_{j=1}^n \left( 1 - \frac{1}{p_j^{d \sigma}} \right)^{-1} \right) \cdot \left( \prod_{j= n+1}^{+\infty} \left( \prod_{\mathfrak{p} \mid p_j \mathcal{O}_{F_{d,n}}}  \left( 1 - \frac{1}{p_j^{f(\mathfrak{p}/p_j) \sigma}} \right)^{-1} \right) \right),
        \]
        where $f(\mathfrak{p}/p_j)$ denotes the inertia index of the prime ideal $\mathfrak{p} \subseteq \mathcal{O}_{F_{d,n}}$ over the prime ideal $p_j \mathbb{Z}$. This shows that $\zeta_{F_{d,n}}(\sigma) \to \zeta(d \sigma)$ when $n \to +\infty$ and $d$ is fixed. 
        More precisely, using the fact that $\left\lvert \log\left( 1 - \frac{1}{x} \right) \right\rvert < \frac{1}{x-1}$ for every $x > 1$, we have that
        \[
            \begin{aligned}
                \left\lvert \log\left\lvert \frac{\zeta_{F_{d,n}}(\sigma)}{\zeta(d \sigma)}\right\rvert \right\rvert &= \left\lvert \sum_{j=n+1}^{+\infty} \left( \log\left\lvert 1 - \frac{1}{p_j^{d \sigma}} \right\rvert - \sum_{\mathfrak{p} \mid p_j \mathcal{O}_{F_{d,n}}} \log\left\lvert 1 - \frac{1}{p_j^{f(\mathfrak{p}/p_j) \sigma}} \right\rvert \right) \right\rvert \leq (d+1) \sum_{j=n+1}^{+\infty} \frac{1}{p_j^\sigma - 1} \\
                &\leq \frac{d+1}{(\sigma - 1) n^{\sigma - 1} \log n },
            \end{aligned}
        \]
        for every $n, d \geq 1$. Setting $\alpha := \lfloor \frac{1}{\sigma-1} \rfloor + 1$ and $K_d := F_{d,d^\alpha}$, we see that
        \[
            \lim_{d \to +\infty} h_\sigma(K_d) = \lim_{d \to +\infty} \zeta_{K_d}(\sigma) = \lim_{d \to +\infty} \zeta(d \sigma) = 1,
        \]
        and therefore that $\inf(h_\sigma(\mathcal{N})) = 1$, as we wanted to show.

        In order to conclude the proof of this proposition, we need to show that $\sup(h_\sigma(\mathcal{N})) = +\infty$. To do so, fix two integers $n,k \geq 1$, and let $q_{1,n},\dots,q_{k,n} \in \mathbb{N}$ be the smallest $k$-many primes such that \[q_{j,n} \equiv 1 \ \text{mod} \ 4 p_1 \dots p_n\] for every $j \in \{1,\dots,k\}$, which exist thanks to Dirichlet's theorem on the existence of infinitely many primes in any given arithmetic progression \cite[Section~2.3]{Iwaniec_Kowalski_2004}. Then, the quadratic fields $\mathbb{Q}(\sqrt{q_{1,n}}),\dots,\mathbb{Q}(\sqrt{q_{k,n}})$ are all linearly disjoint, and the primes $p_1,\dots,p_n$ split in each of these quadratic fields. Therefore, we constructed a number field $L_{k,n} := \mathbb{Q}(\sqrt{q_{1,n}},\dots,\sqrt{q_{k,n}})$ of degree $[L_{k,n} \colon \mathbb{Q}] = 2^k$ such that $p_1,\dots,p_n$ split completely in $L_{k,n}$.
        This implies that
        \[
            \zeta_{L_{k,n}}(\sigma) = \left( \prod_{j=1}^n \left( 1 - \frac{1}{p_j^\sigma} \right)^{-2^k} \right) \cdot \left( \prod_{j= n+1}^{+\infty} \left( \prod_{\mathfrak{p} \mid p_j \mathcal{O}_{L_{k,n}}}  \left( 1 - \frac{1}{p_j^{f(\mathfrak{p}/p_j) \sigma}} \right)^{-1} \right) \right)  
        \]
        and therefore we see that $\zeta_{L_{k,n}}(\sigma) \to \zeta(\sigma)^{2^k}$ when $n \to +\infty$ and $k$ is fixed. More precisely, using once again the fact that $\left\lvert \log\left( 1 - \frac{1}{x} \right) \right\rvert < \frac{1}{x-1}$ for every $x > 1$, we see that
        \[
            \begin{aligned}
                \left\lvert \log\left\lvert \frac{\zeta_{L_{k,n}}(\sigma)}{\zeta(\sigma)^{2^k}} \right\rvert \right\rvert &= \left\lvert \sum_{j=n+1}^{+\infty} \left( 2^k \log\left\lvert 1-\frac{1}{p_j^\sigma} \right\rvert - \sum_{\mathfrak{p} \mid p_j \mathcal{O}_{L_{k,n}}} \log\left\lvert 1 - \frac{1}{p_j^{f(\mathfrak{p}/p_j) \sigma}} \right\rvert \right) \right\rvert \\
                &\leq (2^k+1) \sum_{j=n+1}^{+\infty} \frac{1}{p_j^\sigma - 1} \leq \frac{2^k+1}{(\sigma-1) n^{\sigma-1} \log n}.
            \end{aligned} 
        \]
        Setting $\alpha := \lfloor \frac{1}{\sigma-1} \rfloor + 1$ and $M_k := L_{k,2^{\alpha k}}$, we have that
        \[
            \lim_{k \to +\infty} h_\sigma(M_k) = \lim_{k \to \+infty} \zeta_{M_k}(\sigma) = \lim_{k \to +\infty} \zeta(\sigma)^{2^k} = +\infty,
        \]
        which finally allows us to conclude that $\sup(h_\sigma(\mathcal{N})) = +\infty$.
    \end{proof}
      
    \begin{remark}
        The proof of \cref{s>1} shows that if $K$ is any number field of degree $[K \colon \mathbb{Q}] = d$ inside which the primes $p_1,\dots,p_n$ split completely, then
        \[
            \left\lvert \log\left\lvert \frac{\zeta_K(\sigma)}{\zeta(\sigma)^d} \right\rvert \right\rvert \leq \frac{d+1}{(\sigma-1) n^{\sigma-1} \log n}.
        \]
        In fact, for any fixed $d$ and $n$, such a number field always exists, and can even be constructed as an abelian extension of $\mathbb{Q}$. Indeed, the theorem of Grünwald and Wang \cite[Theorem~9.2.8]{Neukirch_Schmidt_Wingberg_2008} allows one to construct such an extension if $d$ is odd. On the other hand, if $d = 2^k$ we can take $K = L_{k,n} = \mathbb{Q}(\sqrt{q_{1,n}},\dots,\sqrt{q_{k,n}})$ as above. In the general case, one can write $d = d_0 2^k$, fix an abelian extension $K'$ of degree $d_0$ inside which all the primes $p_1,\dots,p_n$ split completely, and then take some primes $\ell_1,\dots,\ell_k$ such that $\ell_j \equiv 1 \ \mod 4 p_1 \cdots p_n$ and $\mathbb{Q}(\sqrt{\ell_1},\dots,\sqrt{\ell_n})$ is linearly disjoint from $K'$. This allows one to take $K = K'(\sqrt{\ell_1},\dots,\sqrt{\ell_n})$.
    \end{remark}

\subsection{Left of the critical strip}\label{negative s}

    In this section, we study the diophantine properties of special values of Dedekind zeta functions for those $s \in \mathbb{C}$ such that $\mathrm{Re}(s) \leq 0$.

    \begin{proof}[Proof of \cref{Re negative}]
    By \cite[Lemma 3.1]{GeLa24}, for every $s$ with $\sigma=\mathrm{Re}(s)>1$ we have that:
    \begin{equation}\label{bound zeta}
    \frac{1}{\zeta(\sigma)^{d_K}}\leq |\zeta_K(s)|\leq \zeta(\sigma)^{d_K},
    \end{equation}
    where $d_K=[K:\Q]$. In addition, by \cite[Lemma 4.3]{GeLa24} we have that
    \begin{equation}\label{bound gamma}
    \Gamma_{m}(s)^{d_K}\leq\left|\frac{\Gamma_\R(1 - s)^{r_1} \Gamma_\C(1 - s)^
{r_2}}{\Gamma_\R(s)^{r_1} \Gamma_\C(s)^{r_2}
}\right|,    
    \end{equation}
    where $\Gamma_m$ is the function defined in \eqref{eq:Gamma_m}.
    By \cite[Theorem 1.2]{Pazuki_Pengo_2024} we know that if $s \in \mathbb{Z}_{\leq 0}$ is a non-positive integer, $h_s$ satisfies the Northcott property. In particular, the function $[K]\mapsto \left(\lvert\zeta_K^\ast(s) \rvert, d_K\right)$ has the Northcott property. Thus, we can assume that $s$ is not a negative integer. Then, we have that $\Gamma_{m}(s)>0$.
    
     Applying \eqref{bound zeta} and \eqref{bound gamma} to the functional equation \eqref{functional equation}, we obtain that for every $s\in \C$ with $\sigma<0$, we have
    \[
|\zeta_K(s)|\geq \left(\frac{\Gamma_{m}(s)}{|\zeta(1-\sigma)|}\right)^{d_K}|\Delta_K|^{\frac{1}{2}-\sigma}.
    \]
    Let $B$ be a positive number. We assume that $d_K\leq B$ and aim to prove that $[K]\mapsto \lvert\zeta_K^\ast(s) \rvert$ has the Northcott property. 
    Since the discriminant function $[K]\mapsto \lvert\Delta_K \rvert$ has the Northcott property, thanks to a celebrated theorem of Hermite's, for all but finitely many isomorphism classes $[K]$, we have that
    \[
    |\Delta_K|^{\frac{1}{4}}\geq \left(\frac{\Gamma_{m}(s)}{|\zeta(1-\sigma)|}\right)^{B}.
    \]
    Therefore, for all but finitely isomorphism classes of number fields $[K]$ with $d_K\leq B$, we have
    \[
|\zeta_K(s)|\geq |\Delta_K|^{\frac{1}{4}-\sigma}.
    \]
    Consequently, since $\sigma< 0$, the function $[K]\mapsto \left(\lvert\zeta_K^\ast(s) \rvert, d_K\right)$ has the Northcott property.
    \end{proof}

    \section*{Acknowledgments}

    We thank Matilde Lalín, Asbjørn Christian Nordentoft, Morten Risager, and Kannan Soundararajan for several useful discussions. Moreover, we thank David Rohrlich for his comments on the first version of this work. Finally, we thank the anonymous referee for their precious comments and suggestions, which helped us improve the expositional quality of the present paper.
    \section*{Funding}
    We thank the CNRS (IRN GANDA) for its support. The first author was supported by Simons Foundation grant no.\ 550023. The second author was supported by ANR {\it Jinvariant} (ANR-20-CE40-0003). He is grateful to the University of Bordeaux for the hospitality.
The third author is grateful to the École normale supérieure de Lyon, to the Max Planck Institute for Mathematics in Bonn and to the Leibniz University in Hannover for providing excellent working conditions, great hospitality, and financial support.
Moreover, the third author thanks the research project ``Motivic homotopy, quadratic invariants and diagonal classes'' (ANR-21-CE40-0015) and the LABEX MILYON (ANR-10-LABX-0070) of the Université de Lyon, within the program
    ``Investissements d’Avenir'' (ANR-11-IDEX-0007), for their financial support.

\begin{wrapfigure}[3]{R}{0.2\linewidth}
\vspace{-0.5cm}
\includegraphics[width=\linewidth]{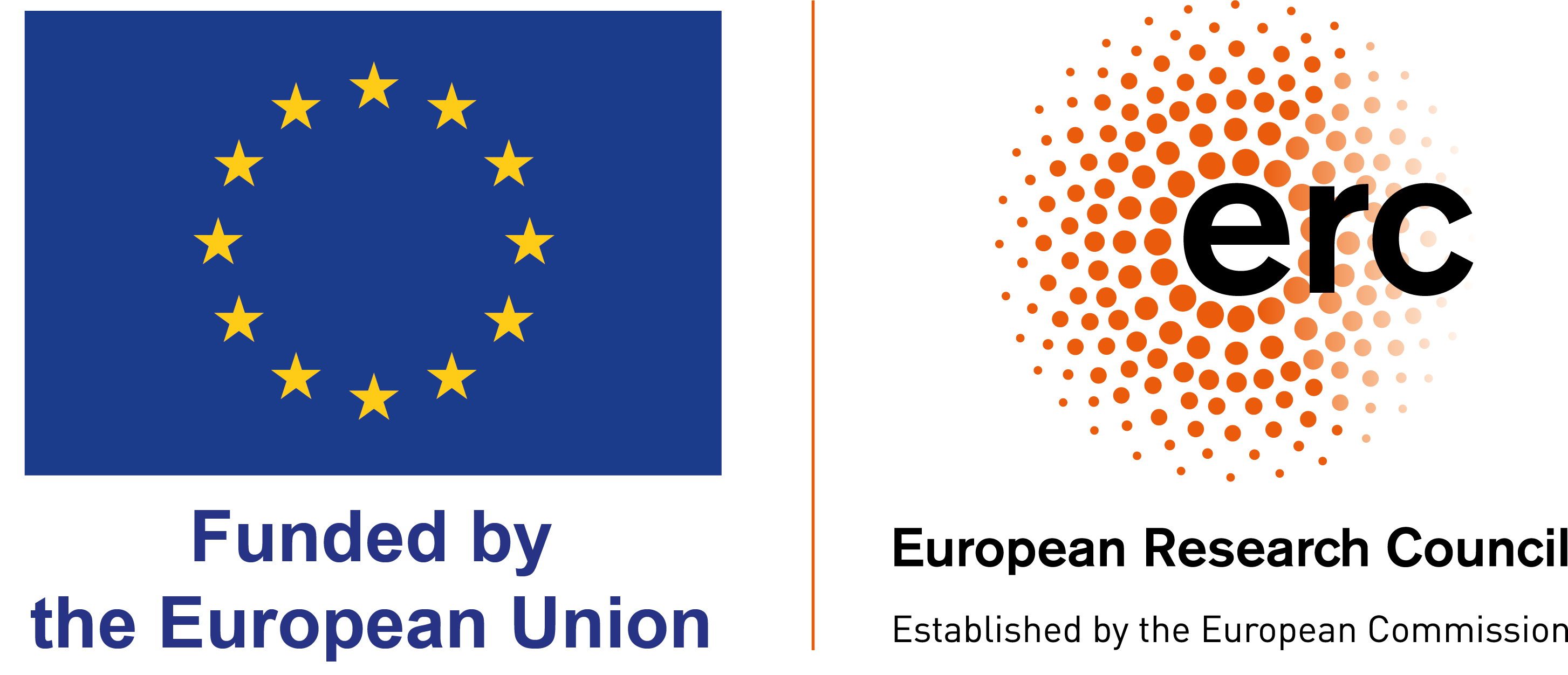} 
\end{wrapfigure}
\vspace{0.2cm}
Furthermore, the third author received funding from the European Research Council (ERC) under the European Union’s Horizon 2020 research and innovation programme (grant agreement number 945714).

\vspace{\baselineskip}
    \section*{Data availability statement}

    Data sharing not applicable to this article as no datasets were generated or analysed during
the current study.
    \section*{Conflict of interest}
    
    On behalf of all authors, the corresponding author states that there is no conflict of
interest.

\printbibliography

\end{document}

%% file: shortcuts.tex
\newcommand{\F}{\mathbb{F}}

\newcommand{\Q}{\mathbb{Q}}
\newcommand{\R}{\mathbb{R}}
\newcommand{\C}{\mathbb{C}}

\renewcommand{\mod}{\operatorname{mod}\ }

\newcolumntype{C}[1]{>{\centering\arraybackslash}p{#1}}

\newcommand\restr[2]{{ \left.\kern-\nulldelimiterspace #1 \vphantom{\big|} \right|_{#2} }}